\documentclass[pdftex,a4paper,oneside]{amsart}

\usepackage[utf8]{inputenc}
\usepackage[USenglish]{babel}
\usepackage[T1]{fontenc}
\usepackage{csquotes}

\usepackage{geometry}
\usepackage{caption}
\usepackage{booktabs}
\usepackage{enumerate}
\usepackage[shortlabels]{enumitem}
\usepackage{calc}
\newlength{\IdentationTDConditions}
\newlength{\IdentationTDExtraCondition}
\usepackage{amsmath}
\usepackage{amstext}
\usepackage{amssymb}
\usepackage{amsfonts}
\usepackage{mathrsfs}
\usepackage{amssymb}
\usepackage{amsthm}
\allowdisplaybreaks
\usepackage{etoolbox}

\usepackage{needspace} % avoids theorem heading to be separated from theorem body

\usepackage{graphics}

\usepackage[norelsize, linesnumbered, ruled]{algorithm2e} %tworuled, 
\SetKw{Return}{Return}
\SetKwIF{If}{ElseIf}{Else}{If}{then}{Else If}{Else}{Endif}
\SetKwFor{While}{While}{do}{Endw}
\SetKw{And}{and}

\renewcommand{\leIf}[3]{\KwSty{If} #1 \KwSty{then} #2 \KwSty{else} #3\;}

\newcommand{\menge}[1]{\mathbb{#1}}
\newcommand{\N}{\menge{N}}

\newcommand{\cupdot}{\ \ensuremath{{\mathaccent\cdot\cup}}\ }
\newcommand{\diam}{\operatorname{diam}}
\newcommand{\bigO}{\operatorname{\mathcal{O}}}
\newcommand{\mb}{\operatorname{MinBis}}

\newcommand{\X}{\mathcal{X}}
\newcommand{\Bceil}[1]{\left\lceil#1\right\rceil}
\newcommand{\Bfloor}[1]{\left\lfloor#1\right\rfloor}
\newcommand{\tw}{\operatorname{tw}}
\newcommand{\notEx}{\operatorname{null}}
\newcommand{\larr}{\leftarrow}

\newcommand{\superscript}[1]{\ensuremath{^{\textrm{#1}}}}

\renewcommand{\th}[0]{\superscript{th}}

\newtheorem{thm}{Theorem}
\newtheorem{lemma}[thm]{Lemma} %enumerated with theorems
\newtheorem{cor}[thm]{Corollary} %enumerated with theorems
\newtheorem{prop}[thm]{Proposition}
\newtheorem{defi}[thm]{Definition}

\makeatletter
\providerobustcmd*{\bigcupdot}{%
  \mathop{%
    \mathpalette\bigop@dot\bigcup
  }%
}
\newrobustcmd*{\bigop@dot}[2]{%
  \setbox0=\hbox{\(\m@th#1#2\)}%
  \vbox{%
    \lineskiplimit=\maxdimen
    \lineskip=-0.7\dimexpr\ht0+\dp0\relax
    \ialign{%
      \hfil##\hfil\cr
      \(\m@th\cdot\)\cr
      \box0\cr
    }%
  }%
}
\makeatother

\begin{document}

\title[Minimum Bisection and Related Cut Problems]{On Minimum Bisection and Related Cut Problems \\in Trees and Tree-Like Graphs} 

\author{Cristina G.\ Fernandes}
\address[Cristina~G.~Fernandes]{Instituto de Matem\'atica e Estat\'{\i}stica \\
  Universidade de S\~ao Paulo\\
  Rua do Mat\~ao~1010, 05508--090\\
  S\~ao Paulo\\
  Brazil}
\email{cris@ime.usp.br}
\thanks{The first author was partially supported by CNPq Proc.~308523/2012-1 and 477203/2012-4, FAPESP 2013/\mbox{03447-6}, and Project MaCLinC of NUMEC/USP}
 
\author{Tina Janne Schmidt}
\author{Anusch Taraz}
\address[Tina~Janne~Schmidt and Anusch~Taraz]{Institut f\"ur Mathematik\\
  TU Hamburg\\
  Am Schwarzenberg-Campus~3E\\
  21073 Hamburg\\
  Germany}
\email{tina.janne.schmidt@tuhh.de \textnormal{and} taraz@tuhh.de}
\thanks{The second author gratefully acknowledges the support by the Evangelische Studienwerk Villigst e.V.\\ 
The research of the three authors was supported by a PROBRAL CAPES/DAAD Proc.~430/15 (February 2015 to December 2016, DAAD Projekt-ID 57143515). \\
Some of the results that are proven in this work have already been announced in extended abstracts at EuroComb2013 and EuroComb2015 \cite{EuroComb2013, EuroComb2015}}

%\author{Anusch Taraz}
%\address[Anusch Taraz]{Institut f\"ur Mathematik, TU Hamburg-Harburg, Am Schwarzenberg-Campus~3E, 21073 Hamburg, Germany}
%\email[Anusch Taraz]{taraz@tuhh.de}

\date{\today}



%\begin{abstract}
%\end{abstract}

\maketitle          %                  Erzeugt die Titelseite
\thispagestyle{empty}

%\vspace{-7pt}

\begin{quote}
\footnotesize
\textsc{Abstract.} Minimum Bisection denotes the NP-hard problem to partition the vertex set of a graph into two sets of equal sizes while minimizing the width of the bisection, which is defined as the number of edges between these two sets. 
We first consider this problem for trees and prove that the minimum bisection width of every tree~\(T\) on~\(n\) vertices satisfies~\({\mb(T) \leq 8n\Delta(T) / \diam(T)}\). 
Second, we generalize this to arbitrary graphs with a given tree decomposition~\( (T,\X)\) and give an upper bound on the minimum bisection width that depends on the structure of~\( (T,\X)\). 
Moreover, we show that a bisection satisfying our general bound can be computed in time proportional to the encoding length of the tree decomposition when the latter is provided as input. 
\end{quote}

\section{Introduction}
\label{secIntroduction}

\subsection{The Minimum Bisection Problem}
\label{subsecMinBisProb}

A \emph{bisection}~\((B,W)\) in a graph~\(G=(V,E)\) is a partition of its vertex set into two sets~\(B\) and~\(W\), called the black and the white set, of sizes differing by at most one. An edge~\(\{x,y\}\) of~\(G\) is \emph{cut} by the bisection~\( (B,W)\) if~\(x\in B\) and~\(y\in W\) or vice versa. The number of edges cut by the bisection~\( (B,W)\) is called the \emph{width} of the bisection and is denoted by~\(e_G(B,W)\). The \emph{minimum bisection width} of~\(G\) is defined as
\[ \mb(G) := \min \{ e_G(B,W) : (B,W) \text{ is a bisection in~\(G\)} \}\] 
and a bisection of width~\(\mb(G)\) in~\(G\) is called a \emph{minimum bisection} in~\(G\). Determining a minimum bisection is a well-known optimization problem that is -- unlike the Minimum Cut Problem -- known to be NP-hard~\cite{GareyJohnsonStockmeyer}. Jansen et al.~\cite{JansenKarpinski} developed a dynamic programming algorithm that computes a minimum bisection in~\( \bigO ( 2^t n^3 )\) time for an arbitrary graph on~\(n\) vertices, when a tree decomposition of width~\(t\) is provided as input. Thus, the problem becomes polynomially tractable for graphs of constant tree-width. Furthermore, Cygan et al.~\cite{CyganLokshtanov} showed that the Minimum Bisection Problem is fixed parameter tractable. Currently, the best known approximation algorithm is due to R\"acke~\cite{Raecke} and achieves an approximation ratio of~\(\bigO (\log n)\) for arbitrary graphs on~\(n\) vertices. For planar graphs, no better approximation ratio is known and it is open whether the Minimum Bisection Problem remains NP-hard when restricted to planar graphs. When the minimum degree of the considered graph is linear, a PTAS for finding a minimum bisection is known~\cite{AroraPTAS}. 
On the other hand, the Minimum Bisection Problem restricted to \mbox{3-regular} graphs is as hard to approximate as its general version~\cite{BermanKarpinski}.  
In the following, we therefore limit our attention to graphs that have small (in fact, constant) maximum degree.

\subsection{Results for Trees}
\label{subsecResultsTrees}

If~\(T\) is a tree on~\(n\) vertices and maximum degree~\(\Delta\), then owing to the existence of a \emph{separating vertex} (i.e., a vertex whose removal leaves no connected component of size greater than~\(n/2\)), we always have \(\mb(T) \leq \Delta \cdot \log_2 n\) and 
it is easy to see that a bisection satisfying this bound can be computed in~\( \bigO(n)\) time. 
Furthermore, the bound is tight up to a constant factor, as the example of a perfect ternary tree~\(T_h\) of height~\(h\) 
with~\(\mb(T_h) \geq h- \log_3 h \) shows (see e.g.~Theorem~7.6 and Theorem~4.11 in~\cite{MScThesisTina}). 

Here, one of our aims is to investigate the structure of bounded-degree trees that have large minimum bisection width. To do so, our first result establishes the following inequality, where the \emph{relative diameter}~\(\diam^*(T)\) of a tree~\(T\) denotes the fraction of vertices of~\(T\) that lie on a longest path~\(P\) in~\(T\), i.e.,~\(|V(P)| / |V(T)|\).

%\begin{thm}[Theorem~1.1 in \cite{EuroComb2013}]
\begin{thm}
  \label{thmTreeCutSpecSizes}
  Every tree~\(T\) on~\(n\) vertices with maximum degree~\(\Delta\) satisfies 
\[ \mb(T) \leq \frac{8 \Delta}{\diam^*(T)}.\] 
  A bisection satisfying this bound can be computed in~\(\bigO(n)\) time.  
\end{thm}

The following corollary is an immediate consequence of Theorem~\ref{thmTreeCutSpecSizes} and shows that bounded-degree trees with large minimum bisection width do not contain a path of linear length.

\begin{cor}
  \label{corPaths}
  For every~\(\Delta \in \N\) and every~\(c > 0\), there is an~\(\alpha > 0\) such that the following holds: \\
  If~\(T\) is a tree on~\(n\) vertices with maximum degree~\(\Delta\) and~\(\mb(T) \geq c \log n\), then~\(T\) does not contain a path of length~\( \alpha n / \log n\) or greater.
\end{cor}

\subsection{Results for General Graphs}
\label{subsecResultsGenGraphs}

Our main goal in this paper is to establish a similar bound as in Theorem~\ref{thmTreeCutSpecSizes} for general graphs with a given tree decomposition~\((T,\X)\). Instead of considering a longest path in the underlying graph, we define a parameter~\(r(T,\X)\) that roughly measures how close the tree decomposition~\( (T,\X)\) is to a \emph{path decomposition}, i.e.,~a tree decomposition~\( (\tilde{T},\tilde{\X})\) where~\(\tilde{T}\) is a path. Consider a graph~\(G\) on~\(n\) vertices and a path decomposition~\( (P,\X)\) of~\(G\) of width~\(t-1\). 
It is easy to see that~\(G\) allows a bisection of width at most~\(t \Delta(G)\) by walking along the path~\(P\) until we have seen~\(n/2\) vertices of~\(G\) in the clusters and then bisecting~\(G\) there. Therefore, we will define~\(r(T,\X)\) so that~\( r(T,\X) = 1\) holds for path decompositions~\( (T,\X)\) and so that~\(r(T,\X)\) decreases when~\( (T,\X)\) looks less like a path decomposition. Let~\(G=(V,E)\) be a graph on~\(n\) vertices and~\( (T,\X)\) be a tree decomposition of~\(G\) with~\(\X = (X^i)_{i \in V(T)}\). For a path~\(P \subseteq T\), we define the \emph{weight}~\(w(P,\X)\) and the \emph{relative weight}~\(w^*(P,\X)\) of~\(P\) with respect to~\(\X\) to be

\[ w(P, \X) := \left| \bigcup_{i \in V(P)} X^i \right| \quad \quad \text{ and } \quad \quad w^* (P,\X) := \frac{w(P, \X)}{\left| \bigcup_{i \in V(T)} X^i \right|} = \frac{1}{n}w(P,\X),\] 
respectively. The \emph{relative weight of a heaviest path} in~\((T,\X)\) is defined as
  
\[ r ( T, \X) \ := \ \max_{\text{\(P\) path in~\(T\)} } \; \; w^*(P, \X) = \ \frac{1}{n} \ \max_{\text{\(P\) path in~\(T\)} } \; \; w(P, \X).\]   

Observe that every tree decomposition~\( (T,\X)\) satisfies \(\frac{1}{n} \leq r(T,\X) \leq 1\). Furthermore, denote by \( \| (T,\X) \| := |V(T)| + \sum_{i\in V(T)} |X^i| \) the \emph{size} of a tree decomposition~\( (T,\X)\) with~\( \X=(X^i)_{i \in V(T)}\) and observe that the size of~\( (T,\X)\) measures the encoding length of~\((T,\X)\). 
We can now state the following strengthening of Theorem~\ref{thmTreeCutSpecSizes} for general graphs. 

\begin{thm}
  \label{thmCutSpecSizes}
  Every graph~\(G\) that allows a tree decomposition~\( (T,\X) \) of width~\(t-1\) satisfies
  
\[ \mb (G) \ \leq \ \tfrac{ 1 }{2} t \Delta(G) \left( \left( \log_2  \tfrac{1}{r(T,\X)} \right)^2 \ + \  9 \log_2  \tfrac{1}{r(T,\X)} \ + \ 8 \right).\] 
  If the tree decomposition~\( (T,\X)\) is provided as input, a bisection satisfying this bound can be computed in~\( \bigO \left(  \| (T, \X) \| \right) \) time. 
\end{thm}

Theorem~\ref{thmCutSpecSizes} strengthens Theorem~\ref{thmTreeCutSpecSizes} in two respects: first, because it applies to graphs with a given tree decomposition and not just to trees, and second because we now get a better upper bound as the term~\(1/r(T,\X)\) that now plays the r\^ole of~\(1/\diam^*(T)\) only contributes logarithmically. If we dispense for a moment with this improvement, the somewhat weaker but more legible upper bound 

\[\mb(G) \leq \frac{8 t \Delta(G) }{r(T,\X)}\]  
can be obtained. 

What structural information does Theorem~\ref{thmCutSpecSizes} yield for graphs with large minimum bisection width? 
Consider a graph~\(G\) on~\(n\) vertices with bounded degree and bounded tree-width.
Analogously to trees and again owing to the existence of small separators, we have~\(\mb(G)=\bigO(\log n)\). 
Now the above inequality implies that if~\(\mb(G)\) is within a constant of this upper bound, then any tree decomposition~\((T,\X)\) of~\(G\) with width~\(\bigO(\tw(G))\) must satisfy~\( r(T,\X) = \bigO(1/\log n)\), i.e.,~\((T,\X)\) is far from being a path decomposition. 

A final remark, concerning the algorithmic aspects of Theorem~\ref{thmCutSpecSizes}: Similarly to Theorem~\ref{thmTreeCutSpecSizes}, the algorithm corresponding to Theorem~\ref{thmCutSpecSizes} is not guaranteed to compute a minimum bisection. However, in certain situations, it will provide a good approximation and its running time, assuming the input tree decomposition~\((T,\X)\) is nonredundant, is bounded by~\(\bigO(nt)\) while the algorithm in~\cite{JansenKarpinski} that computes a minimum bisection runs in~\(\bigO(2^t n^3)\) time, where~\(t-1\) denotes the width of~\((T,\X)\) and~\(n\) denotes the number of vertices of the underlying graph.

We conclude this subsection with the following lemma that relaxes the size requirements on the sets of the cut and gives an upper bound of~\(\bigO(t \Delta(G))\) on the cut width. It is one of the main tools to prove Theorem~\ref{thmCutSpecSizes} and might be of independent interest. 

\begin{lemma}[Approximate Cut] \label{lemmaApproxCut}
   Let~\(G\) be an arbitrary graph on~\(n\) vertices and let~\((T,\X)\) be a tree decomposition of~\(G\) of width at most~\(t-1\). For every integer~\(m \in [n]\) and every~\(0 < c < 1 \), there is a cut~\( (B,W)\) in~\(G\) with~\(c m < |B| \leq m\) and \(e_G(B,W) \leq \Bceil{\log_2 (1/(1-c))} t \Delta(G)\). If the tree decomposition~\( (T,\X)\) is provided, then a cut satisfying these requirements can be computed in~\(\bigO \left(\|(T,\X)\|\right)\) time, where the hidden constant does not depend on~\(c\).
\end{lemma}

\subsection{Further Remarks}
\label{subsecPlanarGraphs}

The results presented in Theorem~\ref{thmTreeCutSpecSizes} and Theorem~\ref{thmCutSpecSizes} do not only hold for bisections but also for cuts~\( (B,W)\) where the size of~\(B\) is specified as input. Moreover, every tree~\(T\) on~\(n\) vertices allows a tree decomposition~\( (T',\X')\) of width one with~\( r(T',\X') \geq \diam^*(T)\) and~\( \| (T',\X')\| = \bigO(n)\). So we can apply Theorem~\ref{thmCutSpecSizes} to~\(T\) and~\( (T',\X')\) to obtain an asymptotic improvement of Theorem~\ref{thmTreeCutSpecSizes} and a corresponding linear-time algorithm.  Furthermore, the constants in this bound can be improved to
 \begin{align*}
  \mb(T) \ \leq \ \frac{\Delta(T)}{2}  \left( \left( \log_2  \tfrac{1}{\diam^*(T)}  \right)^2 + 7 \log_2 \left( \tfrac{1}{ \diam^*(T)} \right) + 6 \right)
\end{align*}  
by applying the techniques used in the proof of Theorem~\ref{thmCutSpecSizes} directly to the considered tree instead of working with a tree decomposition, see \cite{ThesisTina}. Furthermore, the bounds and algorithms presented here can be extended to \mbox{\(k\)-sections}, where the vertex set of a graph has to be partitioned into~\(k\) sets for some integer~\(k\) while minimizing the number of edges between these sets, see~\cite{EuroComb2015, LagosKSec}. Moreover, extensions to bisections and \mbox{\(k\)-sections} in trees with weighted vertices have been considered, see~\cite{BScThesisFabian}.

Last but not least, we remark that somewhat similar questions are being investigated for planar graphs. Every bounded-degree planar graph on~\(n\) vertices has minimum bisection width~\(\bigO( \sqrt{n})\), which can be shown by using the Planar Separator Theorem~\cite{LiptonTarjan}, similarly to using separating vertices to show that every bounded-degree tree has minimum bisection width~\( \bigO( \log n)\). The square grid shows that this bound is tight up to a constant factor. In~\cite{EuroComb2013, ProofPlanar} we show that if~\(G\) is a bounded-degree planar graph and~\(\mb(G) = \Omega( \sqrt{n})\), then~\(G\) has tree-width~\( \Omega(\sqrt{n})\) and hence~\(G\) must contain a~\(k \times k\)~grid as minor with~\(k = \Omega(\sqrt{n})\).

\subsection{Organization of the Paper}
\label{subsecOrganization}

We do not need to prove Theorem~\ref{thmTreeCutSpecSizes} as Theorem~\ref{thmCutSpecSizes} is stronger. Nevertheless, we sketch the proof of Theorem~\ref{thmTreeCutSpecSizes} in Section~\ref{secProofSketchTrees} to introduce the techniques used in the proof of Theorem~\ref{thmCutSpecSizes}. Section~\ref{secPerlTreeDec} discusses some preliminaries for tree decompositions and presents the proof of Lemma~\ref{lemmaApproxCut}. The last two sections contain a full proof for Theorem~\ref{thmCutSpecSizes}. First, in Section~\ref{secExistenceCut}, we prove that a bisection with the desired properties exists and then, in Section~\ref{secAlgoDetails}, we explain how to implement a linear-time algorithm computing such a bisection.

\section{Proof Sketch for Theorem~\ref{thmTreeCutSpecSizes}}
\label{secProofSketchTrees}

In this section, we sketch the proof of Theorem~\ref{thmTreeCutSpecSizes}, which says that in every tree a bisection within a certain bound can be computed in linear time. 

\subsection{Basic Notation}
\label{subsecBasicNotation}

For a real~\(x\) we denote by~\( \Bfloor{x}\) the largest integer~\(i\) with~\(i \leq x\) and by~\( \Bceil{x}\) the smallest integer~\(i\) with~\(i \geq x\). Let~\(\log := \log_2\) and, for an integer~\(n>0\), define~\([n]:= \{1, 2,\ldots, n\}\).  Consider a graph~\(G\). We use~\(V(G)\) and~\(E(G)\) to refer to the vertex and edge set of~\(G\), respectively, and~\(\Delta(G)\) to denote the maximum degree of~\(G\). Here, we always assume that~\(V(G)\) is nonempty and finite. For a set~\(\emptyset \neq S \subseteq V(G)\) we denote by~\(G[S]\) the subgraph of~\(G\) induced by~\(S\) and for a set~\( S \subseteq V(G)\) with~\(S \neq V(G)\) we denote by~\(G-S\) the graph obtained from~\(G\) by removing all vertices in~\(S\) as well as their incident edges. If~\(S=\{v\}\) then we also write~\(G-v\) instead of~\(G-\{v\}\). Furthermore, for a set~\(F \subseteq E(G)\), the graph obtained from~\(G\) by removing each edge in~\(F\) is denoted by~\(G-F\). We say that~\((V_1, \ldots, V_k)\) is a \emph{cut} in a graph~\(G\) if~\(V_1 \cupdot V_2 \cupdot \cdots \cupdot V_k\) is a partition of~\(V(G)\). Analogously to the case in which~\(k=2\), we say that an edge~\(\{x,y\} \in E(G)\) is \emph{cut} by~\((V_1, \ldots, V_k)\) if~\(x \in V_i\) and~\(y\in V_j\) for some~\(i \neq j\). We denote by~\(E_G(V_1, \ldots, V_k)\) the set of edges cut by~\((V_1, \ldots, V_k)\) and define the \emph{width} of~\((V_1, \ldots, V_k)\) as~\(e_G(V_1, \ldots, V_k) := |E_G(V_1, \ldots, V_k)|\). Furthermore, the notion of the \emph{relative diameter} is extended to forests. If~\(G\) is a forest on~\(n\) vertices and~\(G_1,\ldots,G_k\) are the connected components of~\(G\), then~\(\diam^*(G) := (1/n) \ \sum_{i \in [k]} |V(P_i)|\), where~\(P_i\) is a longest path in~\(G_i\) for each~\(i \in [k]\).

\subsection{Two Special Cases}
\label{subsecTreeSpecialCases}

Before starting to discuss the proof of Theorem~\ref{thmTreeCutSpecSizes}, consider the following lemma, that, when setting~\(m = \Bfloor{n/2}\), can be seen as a forerunner to Theorem~\ref{thmTreeCutSpecSizes}.

\begin{lemma}
 \label{lemmaTreeBase}
 For every forest~\(G\) with~\(n\) vertices and~\(\diam^*(G) > 1/2\), and for every~\(m \in [n]\) there is a cut~\((B,W)\) in~\(G\) with~\(|B| =m\) and~\(e_G (B,W) \leq 2\).
\end{lemma}

\begin{proof}
First, note that it suffices to consider the case when~\(G=(V,E)\) is a tree, as otherwise we can add edges until~\(G\) is connected in such a way that the relative diameter is not affected.  Define~\(d:= \diam^*(G)\), let~\(P = (V_P, E_P)\) be a longest path in~\(G\), denote by~\(x_0\) and~\(y_0\) the ends of~\(P\), and note that~\( |V_P| = dn\). For each~\(v\in V_P\), let~\(T_v\) be the connected component of~\(T-E_P\) that contains~\(v\). Label the vertices of~\(G\) with~\(1,2, \ldots, n\) according to the following rules:
\begin{itemize}[leftmargin= \IdentationTDConditions]
  \item For each~\(v \in V_P\), the vertices of~\(T_v\) receive consecutive labels and~\(v\) receives the largest label among those.
  \item For all~\(v, v' \in V_P\) with~\(v \neq v'\), if~\(x_0\) is closer to~\(v\) than to~\(v'\), then the label of~\(v\) is smaller than the label of~\(v'\).
\end{itemize}
From now on, a number that differs by a multiple of~\(n\) from a label is considered to be the same as this label, and each vertex is identified with its label. Define~\(N_m(v) := v+m\). Then, as~\(N_m\) is injective and~\(|V_P| > n/2\), there must be a vertex~\(v \in V_P\) such that~\(N_m(v) \in V_P\). If we define~\(B := \{v+1, v+2, \ldots, v+m\}\) and~\(W:= V\setminus B\), the cut~\( (B,W)\) cuts at most two edges, see also Figure~\ref{figCutTree}a).
\end{proof}

To prove Theorem~\ref{thmTreeCutSpecSizes}, Lemma~\ref{lemmaTreeBase} can be used as the base of an induction that doubles the relative diameter of the considered graph in each round. Next, we will show how one can find a bisection in a tree with relative diameter greater than~\(1/4\) and thereby present some of the techniques for proving Theorem~\ref{thmTreeCutSpecSizes}.  Afterwards we discuss how to adapt the proof to show Theorem~\ref{thmTreeCutSpecSizes}.

\begin{lemma}
 \label{lemmaTreeBase2}
 In every forest~\(G\) with an even number of vertices and~\(\diam^*(G) > 1/4\), there is a bisection~\( (B,W)\) with~\(e_G(B,W) \leq 8 \Delta (G)\). 
\end{lemma}

Before proving Lemma~\ref{lemmaTreeBase2}, we present a tool, which is a simplified version of Lemma~\ref{lemmaApproxCut} for trees with a worse bound on the number of cut edges. If the number of cut edges is estimated more carefully, the bound~\(e_G(B,W) \leq \Bceil{ \log (1/(1-c)) } \Delta(G)\) is obtained. 

\begin{lemma}[Approximate Cut in Forests]
 \label{lemmaApproxCutTree} For every forest~\(G\) on~\(n\) vertices, for every~\(m \in [n]\), and for every~\(0<c<1\), there is a cut~\( (B,W)\) in~\(G\) that satisfies \(e_G(B,W) \leq \Bceil{2c/(1-c)}\Delta(G)\) and~\(cm < |B| \leq m\).
\end{lemma}

\begin{proof}
  Fix an integer~\(m \in [n-1]\) and a tree~\(T\) on~\(n\) vertices with~\(\Delta(T) \geq 3\). The proof is easy to adapt to forests and trivial when~\(\Delta(T) \leq 2\) or~\(m=n\). Root~\(T\) at an arbitrary vertex~\(r\), and for~\({x\in V(T)}\), denote by~\(T_x\) the subtree rooted at~\(x\). By starting at the root and repeatedly descending to the child with the largest subtree, we can find a vertex~\(x\) with~\( |V(T_x)| > m\) and~\( |V(T_{y})| \leq m\) for all children~\(y\) of~\(x\). Using the sets~\(V(T_y)\), one can construct a set~\(\tilde{B}_1\) for a cut~\( (\tilde{B}_1, \tilde{W}_1)\) in~\({T_1:=T}\) such that~\(T_1[\tilde{W}_1]\) is connected, \(e_{T_1} (\tilde{B}_1,\tilde{W}_1) \leq \deg(x)\), and~\({m_1/2 < |\tilde{B}_1| \leq m_1}\) where~\({m_1 := m}\). For~\({s \geq 2}\) define~\(B_{s-1} := \tilde{B}_1 \cupdot \cdots \cupdot \tilde{B}_{s-1}\) as well as~\(m_s := m-|B_{s-1}|\) and apply the same strategy recursively to~\(T_{s}:=T[\tilde{W}_{s-1}]\). Let~\(s^*\) be the smallest~\(s\) for which~\(|B_s| > cm\). For all~\(s \in [s^*]\), we have~\(|\tilde{B}_s| >  m_s /2 \geq m_{s^*} /2 \geq (1-c)m/2 \) and hence~\(s^* \leq \Bceil{cm / ((1-c)m/2)} = \Bceil{ 2c/(1-c)}\), which shows the desired bound on~\(e_T(B,W)\) for~\(B:= B_{s^*}\) and~\(W:= V(T) \setminus B\).
\end{proof}

\begin{figure}
 \mbox{\includegraphics{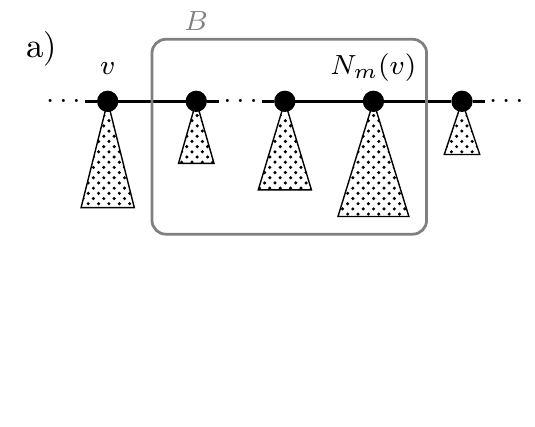}\includegraphics{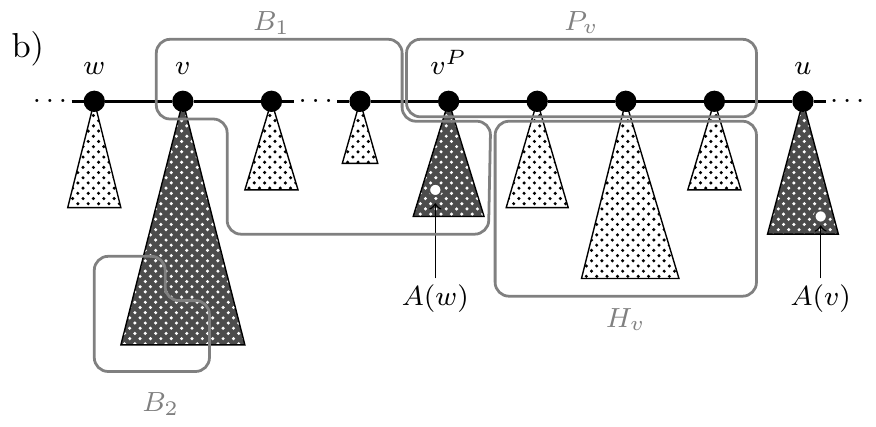}}
 \caption{Proof of Lemma~\ref{lemmaTreeBase} and Lemma~\ref{lemmaTreeBase2}. In both parts, the path~\(P\) is drawn on the top and the trees~\(T_x\) for~\(x \in V_P\) are hanging down from the path~\(P\). a)~Lemma~\ref{lemmaTreeBase}.  b)~Lemma~\ref{lemmaTreeBase2}: Notation and construction of the bisection in Case~2. A tree~\(T_y\) is colored dark gray if it contains a vertex~\(x\) with~\(A(x) \in V_P\), i.e., if the vertex~\(y\) is special.}
 \label{figCutTree}
\end{figure}

\begin{proof}[Proof of Lemma~\ref{lemmaTreeBase2}]
Let~\(G=(V,E)\) be such a forest and denote by~\(n\) its number of vertices. First, note that we may assume that~\(\Delta(G) \geq 3\) and~\(n \geq 4\) as otherwise~\(\diam^*(G) > 1/2\) and Lemma~\ref{lemmaTreeBase} applies. Moreover, it suffices to consider the case in which~\(G\) is a tree, as otherwise we can add edges until~\(G\) is connected without affecting the relative diameter. Denote by~\(P=(V_P,E_P)\) a longest path in~\(G\) and note that~\(|V_P| > n/4\). Define~\(T_v\) for all~\(v\in V_P\), \(x_0\), \(y_0\), and the vertex labeling as in the proof of Lemma~\ref{lemmaTreeBase}. Furthermore, let~\(T_v' := V(T_v)\setminus\{v\}\) for all~\(v\in  V_P\) and denote by~\(A(v) := v + n/2\) the~\emph{antipole} of~\(v\) for each vertex~\(v \in V\). 

\textbf{Case~1:} There is a vertex~\(v \in V_P\) with~\(A(v) \in V_P\). 

Then choosing~\(B:=\{v+1, \ldots, v+n/2\}\) and~\(W:=V \setminus B\) gives a bisection of width at most~\(2\), which is similar to the cut used in Lemma~\ref{lemmaTreeBase}.

\pagebreak[3]

\textbf{Case~2:} For every~\(v \in V_P\), we have~\(A(v) \not\in V_P\).

First, for sets~\(U \subseteq V\), let~\(A(U) := \{ A(u): u \in U\}\). We call a vertex~\(v \in V_P\) \emph{special} if there is a vertex~\(x \in T_v'\) with~\(A(x) \in V_P\). For every special vertex~\(v \in V_P\), define~\(P_v := A(T_v') \cap V_P\). Note that~\(P_v\) induces a path or a cycle in the graph obtained from~\(G\) by adding the edge~\( \{x_0, y_0\}\), and that~\( \{P_v : \text{\(v \in V_P\) is special}\}\) is a partition of~\(V_P\). Also, for every special vertex~\(v \in V_P\), let~\(v^P = A(x)\) for the smallest vertex~\(x \in T_v'\) with~\({A(x) \in V_P}\) and let~\(H_v\) be the union of the sets~\(T_{x}'\) for all~\(x \in P_v\) with~\(x \neq v^P\). See also Figure~\ref{figCutTree}b) for a visualization of these definitions. By construction we have
 \begin{align} 
 \label{proofLemmaBase2_Inclustion}
 \left| T_v ' \right| \geq \left|P_v \vphantom{'} \right| + \left| H_v \vphantom{'} \right|
\end{align}  
for all special~\(v\in V_P\) as well as
 \begin{align} 
 \label{proofLemmaBase2_Counting} 
 \sum_{\substack{v \in V_P: \\ \text{\(v\) is special}}} \left|P_v \right| \; = \; \left| V_P \right| \; > \; \frac{1}{4} \ n  \quad \quad \text{ and } \quad  \frac{3}{4}\ n \; >  \ \left| V \setminus V_P \vphantom{T_{v^P}'} \right| \ = \ \sum_{\substack{v \in V_P: \\ \text{\(v\) is special}}} \left( \left|H_v \vphantom{_{v^P}'} \right| + \left|T_{v^P} '\right|\right).
\end{align}  
Using that~\(A(A(x)) = x\), one can show that a vertex~\(v \in V_P\) is special if and only if there is a special vertex~\(w \in V_P\) such that~\(v = w^P\). Therefore,

\[ \sum_{ \substack{w\in V_P: \\ \text{\(w\) is special}}} \left| T_{w^P} ' \right| \; \;= \; \sum_{\substack{v\in V_P: \\ \text{\(v\) is special }}} \left| T_v ' \vphantom{T_{w^P} '}\right|, \] 
which together with~\eqref{proofLemmaBase2_Counting} implies
 \begin{align*}
 3 \sum_{\substack{v \in V_P: \\ \text{\(v\) is special}} } \left| P_v \right| \; \;  > \; \; \frac{3}{4}n \; \; > \; \;  \sum_{\substack{v \in V_P: \\ \text{\(v\) is special}} } \left( \left| H_v \vphantom{T_{v^P}'} \right| + \left|T_{v^P}'\right| \right) \; \;  = \; \;  \sum_{\substack{v \in V_P: \\ \text{\(v\) is special}} } \left( \left| H_v \vphantom{T_{v}'} \right| + \left|T_{v}' \right| \right). 
\end{align*}  
Consequently, there is a special vertex~\(v \in V_P\) such that
 \begin{align} 
  \label{proofLemmaBase2_SpecialVertex}
  \left|T_v' \right| + \left|H_v \vphantom{T_v'} \right| < 3 \ \left|P_v \right| .
\end{align}  
Replacing~\(|T_v'|\) with~\eqref{proofLemmaBase2_Inclustion}, we obtain \( |P_v| + 2|H_v| < 3 |P_v|\), or equivalently \( |P_v| + |H_v| <  2 |P_v|\). Therefore, the set~\(Z := P_v  \cupdot H_v\) satisfies~\(|Z| < 2 |P_v|\) and also \(Z \neq \emptyset\)  because~\(P_v \neq \emptyset\) as~\(v\) is special. This implies that \(\diam^*(G[Z]) \geq |P_v| / |Z| > 1/2\). The graph~\(G[Z]\) is useful as we can apply Lemma~\ref{lemmaTreeBase} to cut off exactly as many vertices as we like while cutting few edges, but it is too small to cut off the entire set~\(B\) from~\(G[Z]\). 

Next, we will use the vertex~\(v\) and the set~\(Z\) to construct the set~\(B\) for the bisection from three disjoint parts~\(B_1 \subseteq V\setminus (T_v ' \cup Z) \), \(B_2 \subseteq T_v'\), and~\(B_3 \subseteq Z\), see also Figure~\ref{figCutTree}b) for a visualization of the following definitions. If~\(v = v^P\), let~\(B_1 := \emptyset\), and otherwise let~\(B_1 := \{v, v+1, \ldots, v^P -1\}\). As~\(|B_1 \cup Z| < m\) might happen, we cut off a set~\(B_2\) from~\(T_v'\) to ensure that~\(|B_1 \cup B_2 \cup Z| \geq m\). To do so, define~\(m_2 := n/2 - |B_1|\) and observe that~\(1 \leq m_2 \leq |T_v '|\) as~\(v^P -n/2\) is in~\(T_v'\). Next, we apply Lemma~\ref{lemmaApproxCutTree} to~\(G[T_v']\) with parameters~\(m_2\) and~\(c= 2/3\) to obtain a cut~\((B_2, W_2)\) in~\(G[T_v']\) with~\(2 m_2 /3 \leq |B_2| \leq m_2\) and~\(e_{G[T_v']} (B_2, W_2) \leq 4 \Delta(G) \). Let~\(m_3 := n/2 -|B_1| - |B_2|\) and note that~\(B_1\), \(B_2\), and~\(Z\) are pairwise disjoint. Then,
 \begin{align*}
 m_3 \ = \ (\tfrac{1}{2}n -|B_1|) - |B_2| \ \leq \ m_2 - \tfrac{2}{3} m_2 \ \leq \ \tfrac{1}{3} |T_v'| \ < \ |P_v| \ \leq \ |Z|,
\end{align*}  
where we used~\eqref{proofLemmaBase2_SpecialVertex} to derive the second to last inequality. If~\(m_3 =0\), let~\(B_3 := \emptyset\). Otherwise, it is feasible to apply Lemma~\ref{lemmaTreeBase} with parameter~\(m_3\) to~\(G[Z]\) to obtain a cut~\( (B_3,W_3)\) in~\(G[Z]\) with~\(e_{G[Z]} (B_3,W_3) \leq 2 \leq \Delta(G)\) and~\(|B_3| = m_3\). Defining~\(B:= B_1 \cupdot B_2 \cupdot B_3\) and~\(W := V\setminus B\), we obtain a bisection in~\(G\). 

Next, we estimate the number of edges cut by~\( (B,W)\) in~\(G\). Let~\( \tilde{V} := V \setminus (T_v' \cup B_1 \cup Z )\) and denote by~\(u \in V_P\) the vertex with~\(A(v) \in T_u'\). The cut~\( (T_v', B_1, Z, \tilde{V})\) cuts only edges incident to~\(v\), \(v^P\), and~\(u\), i.e., at most~\(3 \Delta(G)\) edges. Using the bounds on the number of cut edges of the cut~\( (B_2, W_2)\) in~\(G[T_v']\) and the cut~\( (B_3, W_3)\) in~\(G[Z]\), the desired bound on~\(e_G(B,W)\) is obtained.
\end{proof}

\subsection{Proof Sketch for Theorem~\ref{thmTreeCutSpecSizes}}
\label{subsecTreeThm}

Although the minimum bisection problem asks for a vertex partition into two classes of (almost) the same size, we need to take a more general approach and consider partitions where the sizes of the classes can be specified by an input parameter~\(m\) in order to be able to apply induction. The following lemma is the heart of the induction to prove Theorem~\ref{thmTreeCutSpecSizes}. Its main idea of doubling the relative diameter also appears in the proof of Lemma~\ref{lemmaTreeBase2}: If we cannot easily find a cut~\((B,W)\) in~\(G\) with~\(|B| =m \) and~\(e_G(B,W) \leq 2\), more precisely, if Case~1 does not apply, then we construct a subgraph~\(G[Z]\) with~\(\diam^*(G[Z]) \geq 2 \diam^*(G)\). 

\begin{lemma}
  \label{lemmaDoubleDiam}
  For every forest~\(G\) on~\(n\) vertices and for every~\( m \in [n]\), the vertex set of~\(G\) can be partitioned into three classes~\(B\),~\( W\), and~\(Z\) such that one of the following two options occurs: 
  \begin{enumerate}[1)]
    \item\label{lemmaDoubleDiamOpt1}  \( \left| B \right| =m \), \ \(Z=\emptyset\), \ and \ \( e_G (B, W, Z) \leq 2\), or
    \item\label{lemmaDoubleDiamOpt2}   \( \left| B \right| \leq m \leq \left| B \right| + \left| Z \right|\), \  \(0 < |Z| \leq n/2\),  \ \( e_G (B, W, Z) \leq 2 \Delta(G) / \diam^* (G) \), \ as well as\  \( \diam ^* (G [Z]) \geq 2 \diam^*(G)\). 
  \end{enumerate}
\end{lemma}

The lemma says that we can either find a partition into two sets~\(B\) and~\(W\) with exactly the right cardinality by cutting very few edges, or there is a partition with an additional set~\(Z\) such that the set~\(B\) is smaller and the set~\(B\cupdot Z\) is larger than the required size~\(m\), as well as the additional feature that the relative diameter of~\(G[Z]\) is at least twice as large as that of~\(G\). Applying Lemma~\ref{lemmaDoubleDiam} recursively to the graph~\(G':=G[Z]\) with parameter~\(m':=m-|B|\), the relative diameter is doubled in each round, until it exceeds~\(1/2\) and we can complete the proof by applying Lemma~\ref{lemmaTreeBase}. Note that as the relative diameter of~\(G\) increases, the bound on the number of cut edges in Option~\ref{lemmaDoubleDiamOpt2} decreases. 

Let us conclude this section with a few words on how to generalize the proof of Lemma~\ref{lemmaTreeBase2} to obtain Lemma~\ref{lemmaDoubleDiam}. We use the same assumptions and notation as in the proof of Lemma~\ref{lemmaTreeBase2}. Here, we consider cuts~\((B,W)\) with~\(|B| = m\) for some input parameter~\(m\). So, instead of working with the antipole of a vertex, we use~\(N_m(v) := v+m\). If there is a vertex~\(v \in V_P\) with~\(N_m(v) \in V_P\), we use the cut with~\(B=\{v+1, \ldots, v+m\}\), similar to Case~1 in the proof of Lemma~\ref{lemmaTreeBase2}. So from now on, assume that this is not the case. Note that~\(N_m(N_m(v)) =v\) is not necessarily true anymore. Also, when a vertex~\(v\) is special with respect to~\(N_m\), there is not necessarily a vertex~\(w\) with~\(v= w^P\). To overcome this, we work with a forward and a backward version of being special. We say a vertex~\(v \in V_P\) is \mbox{\emph{b-special}} if there is a vertex~\(x\in T_v'\) with~\({x-m \in V_P}\) and~\(v\) is \mbox{\emph{f-special}} if there is a vertex~\(x \in T_v'\) with~\(x+m \in V_P\). We use b as in backward and f as in forward, because going~\(m\) steps backward and~\(m\) steps forward from~\(x\) gives a vertex in~\(R\), respectively. Also, we define a forward and a backward version of~\(P_v\) and~\(H_v\), called~\(P_v^b\), \(P_v^f\), \(H_v^b\), and \(H_v^f\). The accounting, which still uses the same ideas, becomes more involved and shows that one of the following must exist:
\begin{enumerate}[a)]
  \item a \mbox{b-special} vertex~\(v \in V_P\) with~\( |T_v'| + |H_v ^b| \leq \left(\tfrac{1}{d}-1 \right) |P_v^b|\), or
  \item an \mbox{f-special} vertex~\(v \in V_P\) with~\( |T_v' | + |H_v^f| \leq \left(\tfrac{1}{d}-1\right) |P_v^f|\),
\end{enumerate}
where~\(d:= \diam^*(G)\). The construction of the cut~\((B,W,Z)\) in each case is then similar to the cut~\( (B_1 \cup B_2, W', Z)\) in Lemma~\ref{lemmaTreeBase2}, where~\(W':= V \setminus (B_1 \cup B_2 \cup Z)\). We define~\(Z:= P_v^b \cupdot H_v^b\) or~\(Z:= P_v^f \cupdot H_v^f\) and show that~\(\diam^*(G[Z]) \geq 2d\). To obtain the set~\(B_2 \subseteq T_v'\), we apply Lemma~\ref{lemmaApproxCutTree} with a parameter~\(c\) that depends on~\(d\), to match the inequalities derived from the accounting. 

Before going on with general graphs and tree decompositions, let us quickly discuss some algorithmic aspects of Theorem~\ref{thmTreeCutSpecSizes}. To implement Lemma~\ref{lemmaDoubleDiam}, a longest path in a tree can be computed in linear time using a procedure similar to the one in~\cite{Handler}, where a center vertex of a tree is computed, which is in the middle of a longest path in a tree. Lemma~\ref{lemmaDoubleDiam} can be implemented in a way that the cut~\((B,W,Z)\) and the subgraph~\(G[Z]\), if Option~\ref{lemmaDoubleDiamOpt2} occurs, are computed in~\( \bigO(n)\) time. As the number of vertices in the considered graph shrinks by at least~\(1/2\) in each round, a bisection satisfying the bound in Theorem~\ref{thmTreeCutSpecSizes} can be computed in linear time.

\section{Preliminaries for Tree Decompositions and Proof of Lemma~\ref{lemmaApproxCut}}
\label{secPerlTreeDec}

In this section, we present some preliminaries for tree decompositions and discuss basic properties that allow us to generalize the ideas used for trees in the previous section. Furthermore, the existence of a cut with the properties in Lemma~\ref{lemmaApproxCut} is presented. A discussion of how to implement the corresponding algorithm can be found in Appendix~\ref{subsubsecImplApproxCut}.

\subsection{Tree Decompositions and Cuts}
\label{subsecTreeDecCuts}

Let us start with the formal definition of a tree decomposition.

\begin{defi}\label{defTreeDec}
  Let~\(G\) be a graph, \(T\) be a tree, and~\(\X= (X^i)_{i \in V(T)}\) with~\(X^i \subseteq V(G)\) for every~\({i \in V(T)}\). The pair~\((T,\X)\) is a \emph{tree decomposition} 
  of~\(G\) if the following three properties 
  hold.
  \begin{enumerate}[(T1), leftmargin= \IdentationTDConditions]
    \item For every~\(v \in V(G)\), there is some~\(i \in V(T)\) such that~\(v \in X^i\).
    \item For every~\(e \in E(G)\), there is some~\(i \in V(T)\) such that~\(e \subseteq X^i\). 
    \item For all~\(i, j \in V(T)\) and all~\(h \in V(T)\) on the (unique)~\(i\text{,}j\)-path in~\(T\), we have~\(X^i \cap X^j \subseteq X^h\). 
  \end{enumerate}  
  The \emph{width} of~\((T,\X)\) is defined as \(\max \{|X^i|-1 : i \in V(T)\}\). 
  The \emph{tree-width}~\(\tw(G)\) of~\(G\) is the smallest integer such that~\(G\) allows a tree decomposition of width~\(\tw(G)\). 
\end{defi}

Consider a graph~\(G=(V,E)\) and a tree decomposition~\( (T,\X)\) with~\( \X= (X^i)_{i\in V(T)}\) of~\(G\). To distinguish the vertices of~\(G\) from the vertices of~\(T\) more easily, we refer to the vertices of~\(T\) as \emph{nodes}. Furthermore, for~\(i \in V(T)\), we refer to the set~\(X^i\) as the \emph{cluster} of~\( (T,\X)\) that corresponds to the node~\(i\), or simply the cluster of~\(i\) when the tree decomposition is clear from the context. It is easy to show that (T3) is equivalent to the following condition.
\begin{enumerate}[\emph{(T3')}, leftmargin= \IdentationTDExtraCondition]
 \item \emph{For every~\(v \in V\), the graph~\(T[I_v]\) is connected, where~\(I_v := \{i \in V(T) : v \in X^i\}\).}
\end{enumerate}

Consider a graph~\(G_0\) and a tree decomposition~\((T_0,\X_0)\) with~\(\X_0 = (X_0^i)_{i \in V(T_0)}\). In order to reapply a procedure to a subgraph~\(G \subseteq G_0\), it is necessary to construct a tree decomposition of~\(G\), which we do in the following way. For a tree~\(T\) with~\(V(T)\subseteq V(T_0)\) and~\(\X= (X^i)_{i\in V(T)}\), we call~\((T,\X)\) the \emph{restriction of~\((T_0,\X_0)\) to~\(T\) and~\(G\)} if~\(X^i = X_0^i \cap V(G)\) for all~\(i\in V(T)\). Note that the restriction of~\((T_0,\X_0)\) to~\(T\) and~\(G\) is not necessarily a tree decomposition of~\(G\), but it will be if we choose~\(T\) and~\(G\) well. For example, it is easy to see that if we choose~\(T=T_0\) and use an arbitrary subgraph~\(G\subseteq G_0\), then the restriction of~\((T_0,\X_0)\) to~\(T\) and~\(G\) is always a tree decomposition of~\(G\). Observe that the width and the size of the restriction of~\((T_0,\X_0)\) to~\(T\) and~\(G\) are at most the width and the size of~\((T_0,\X_0)\), respectively. Usually some clusters of a restriction of~\((T_0,\X_0)\) are empty, which can be avoided with the following concept. A tree decomposition~\( (T,\X)\) of a graph~\(G\) with~\( {\X = (X^i)_{i \in V(T)} } \) is called~\emph{nonredundant} if~\(X^i \not \subseteq X^j\) and~\(X^j \not \subseteq X^i\) for every edge~\( \{i,j\}\) in~\(T\). It is easy to see that any tree decomposition~\((T,\X)\) can be made nonredundant by contracting edges of~\(T\) without increasing its width.

In Section~\ref{secProofSketchTrees}, where we constructed a bisection in a tree~\(\tilde{T}\), we used the following cuts to partition the vertex set of~\(\tilde{T}\) and subtrees of~\(\tilde{T}\) in the proofs of Lemma~\ref{lemmaTreeBase2} and Lemma~\ref{lemmaApproxCutTree}. For some vertex~\(v \in V(\tilde{T})\) we removed all edges incident to~\(v\) and combined the vertex sets of the resulting connected components to obtain a cut in~\(\tilde{T}\). Each time such a construction was used, at most~\(\deg(v) \leq \Delta(\tilde{T})\) edges were cut. This can generalized by considering clusters of a tree decomposition, as done in the next lemma. It uses the following notation: Consider a graph~\(G\) and a tree decomposition~\((T,\X)\) of~\(G\). For each node~\(i\) in~\(T\) we denote by~\(E_{G}(i)\) the set of edges~\(e \in E(G)\) such that~\(e \cap X^i \neq \emptyset\), where~\(X^i\) is the cluster of~\(i\). Furthermore, define~\(e_G(i) := |E_G(i)|\) for every~\(i \in V(T)\) and note that~\(e_G(i) \leq t \Delta(G)\) for every~\(i \in V(T)\), where~\(t-1\) denotes the width of~\( (T,\X)\).  We say that two subgraphs~\(H_1\subseteq G\) and~\(H_2 \subseteq G\) are \emph{disjoint parts} of~\(G\) if~\(V(H_1) \cap V(H_2)= \emptyset\) and there is no edge~\(e=\{x,y\}\) in~\(G\) with~\(x \in V(H_1)\) and~\(y \in V(H_2)\). Note that, if~\(G\) is not connected, then two distinct connected components of~\(G\) are disjoint parts of~\(G\), but the subgraph~\(H_i\) for~\(i \in \{1,2\}\) in the definition of disjoint parts does not have to be connected. The next lemma says that if we remove the edges in~\(E_G(i)\) for some~\(i \in V(T)\), then the graph~\(G\) splits into several disjoint parts. So we can combine these disjoint parts in an arbitrary way to obtain a cut in~\(G\) that cuts at most~\(e_G(i) \leq t \Delta(G) \) edges. The lemma is a widely known fact about tree decompositions, whose proof can be found in~\cite{KleinbergTardos} for example.

\begin{lemma}
  \label{lemmaRemoveClusterEdges}
  Let~\(G=(V,E)\) be an arbitrary graph and let~\( (T, \X)\) be a tree decomposition of~\(G\) with~\({ \X = (X^i)_{i \in V(T)} } \). Fix an arbitrary node~\(i \in V(T)\), let~\(k := \deg_T (i) \) and denote by~\( i_1, i_2, \ldots, i_k\) the neighbors of~\(i\) in~\(T\). Furthermore,  for~\( \ell \in [k]\), let~\(V_{\ell}^T\) be the node set of the component of~\(T-i\) that contains~\(i_\ell\). Removing the edges in~\(E_G(i)\) from~\(G\) decomposes~\(G\) into~\(k + |X^i| \) disjoint parts, which are~\( (\{v\}, \emptyset) \) for every~\(v \in X^i\) and~\(G[V_\ell] \) for every~\( \ell \in [k]\), where~\( V_\ell \ := \bigcup_{j \in V_\ell ^T} X^j \setminus X^i\).
\end{lemma}

\subsection{Approximate Cuts of Small Width}
\label{subsecApproxCut}

Here, we prove the existence part of Lemma~\ref{lemmaApproxCut} that relaxes the size requirements and therefore produces an approximate cut, which means the following. We would like to cut a graph~\(G\) on~\(n\) vertices with a given tree decomposition~\( (T,\X)\) into two pieces~\(B\) and~\(W\), where~\(|B| = m\) for some fixed~\(m \in [n]\), but this might require to cut many edges, in particular when nothing is known about the structure of~\( (T,\X)\). Therefore, we replace the size requirements on~\(B\) by~\(cm \leq |B| \leq m\) for an arbitrary~\(0 < c < 1 \), always aiming to cut few edges. Before starting with the proof, the existence part of Lemma~\ref{lemmaApproxCut} is repeated. %The proof of Lemma~\ref{lemmaApproxCut} uses the same ideas as the proof sketch of Lemma~\ref{lemmaApproxCutTree}, but instead of partitioning the graph in each round according to the connected components that are created by removing all edges incident to some vertex, we now remove all edges in~\(E_G(i)\) for some node~\(i\) of the tree decomposition to partition the vertex set according to Lemma~\ref{lemmaRemoveClusterEdges}. 

\begin{lemma}[Existence part of Lemma~\ref{lemmaApproxCut}] \label{lemmaApproxCutExistence}
   Let~\(G\) be an arbitrary graph on~\(n\) vertices and let~\((T,\X)\) be a tree decomposition of~\(G\) of width at most~\(t-1\). For every integer~\(m \in [n]\) and every~\(0 < c < 1 \), there is a cut~\( (B,W)\) in~\(G\) with~\(c m < |B| \leq m\) and \(e_G(B,W) \leq \Bceil{\log (1/(1-c))} t \Delta(G)\). 
\end{lemma}

\begin{proof}
  Let~\(V=V(G)\), \(T=(V_T,E_T)\), and \(\X = (X^i)_{i \in V_T}\). The idea is to build the set~\(B\) iteratively, similar to the proof of Lemma~\ref{lemmaApproxCutTree}. In each step, we choose a node~\(i\) of~\(T\), remove its cluster from the graph~\(G\), put some of the resulting parts of~\(G\) (as described in Lemma~\ref{lemmaRemoveClusterEdges}) into the set~\(B\), and choose one part in which we repeat this procedure, see Algorithm~\ref{algoApproxCut}. The algorithm roots the tree~\(T\) at an arbitrary node~\(r\) and uses the following definition, where~\(p(i)\) denotes the parent of~\(i\) for every~\(i\neq r\) in~\(T\). For~\(i \in V_T\), define  
  
\[ Y^i := \bigcup_{j\text{ descendant of } i} \!\! X^j \quad \quad \text{ and,  for~\(i \neq r\),}  \quad \quad  \quad  \tilde{Y}^i := Y^i \setminus X^{p(i)}. \] 
  Moreover, let~\(\tilde{Y}^r := Y^r\), and set~\(y_i := |Y^i|\) as well as~\(\tilde{y}_i := |\tilde{Y}^i|\) for every~\(i \in V_T\). Applying Lemma~\ref{lemmaRemoveClusterEdges} with an arbitrary node~\(i\) in~\(T\) gives
   \begin{align}  
    & Y^i \ = \ X^i \cupdot \left( \bigcupdot_{\text{\(j\) child of~\(i\)}} \tilde{Y}^j \right) \label{proofApproxCutEqYiPart} \quad \quad \text{and}\\
    & E_G ( Z_1, Z_2, \tilde{Y}^{j_1}, \tilde{Y}^{j_2}, \ldots, \tilde{Y}^{j_k}, V \setminus Y^i ) \ \subseteq \ E_G(i), \label{proofApproxCutEqYiCut} 
  \end{align}  
  where~\(j_1, j_2, \ldots, j_k\) are the children of~\(i\) and~\(Z_1 \cupdot Z_2\) is an arbitrary partition of~\(X^i\). 

  \begin{algorithm}
    \linespread{1}\selectfont
    \caption{computes an approximate cut.}
    \label{algoApproxCut}

    \KwIn{a tree decomposition~\((T,\X)\) of a graph~\(G\) on~\(n\) vertices, an integer~\(m \in [n]\), and a real~\(0 < c  < 1\).} 
    \KwOut{a cut~\((B,W)\) such that \(cm < |B| \leq m\).}

    \medskip

    Root~\(T\) at an arbitrary node~\(r\)\label{proofApproxCutPrepro0}\;
    Find a node~\(i^*\) such that~\(y_{i^*} \geq m\) and~\(y_j < m\) for all descendants~\(j \neq i^*\) of~\(i^*\)\label{proofApproxCutPrepro1}\;
    \(B \leftarrow \emptyset\), \ \(i \leftarrow i^*\)\label{proofApproxCutBInitializiation}\;
    \While{\(|B| \leq cm\)\label{proofApproxCutOuterWhileLoop}}{
      Let~\(k\) be the number of children of~\(i\) and let \((j_1, j_2,\ldots, j_k)\) be the list of children of~\(i\) ordered so that~\(\tilde{y}_{j_1}\geq \tilde{y}_{j_2} \geq \ldots \geq \tilde{y}_{j_k}\)\;\label{proofApproxCutK}
      \leIf{\(k \geq 1\)}{let~\(\ell\) be the largest integer in~\([k]\) with \(\sum_{h=1}^\ell \tilde{y}_{j_h} \leq m-|B|\)}{\(\ell \larr 0\)\label{proofApproxCutEll}}
      \(B \larr B \cup (\bigcup_{h=1}^\ell\tilde{Y}^{j_h})\)\label{proofApproxCutUpdateB}\;\label{proofApproxCutFirstEndIf}
      \eIf{\( \ell =k \)}{
        Let~\(Z\subseteq X^i\) with~\(|Z| = m - |B|\) and \(B \leftarrow B \cup Z\)\; \label{proofApproxCutSetZ}
      }{
        \(j \larr j_{\ell +1}\)\;\label{proofApproxCutChildEll}
        \While{\(j \neq \notEx\) \And \(y_j \geq m - |B|\)\label{proofApproxCutGoDown}}{
          \(i \leftarrow j\)\;\label{proofApproxCutResetI}
          \leIf{\(i\) has a child}{let~\(j\) be a child of~\(i\) with maximal~\(\tilde{y}_j\)}{\(j \larr \notEx\)}
        }\label{proofApproxCutGoDownEnd}
      }\label{proofApproxEndElse}
    }\label{proofApproxCutEndOuterWhileLoop}
    \Return{\( (B, V(G) \setminus B) \)}\;
  \end{algorithm}

  To state some invariants, let~\(s^*\) be the total number of executions of the outer while loop of Algorithm~\ref{algoApproxCut}. For each~\(s\in [s^*] \cup \{0\}\), denote by~\(B_s\) and~\(i_s\) the set~\(B\) and the node~\(i\) after the~\(s\)\th{}~execution of the outer while loop, respectively. For every~\(s\in\{0\} \cup [s^*]\), we have
  \begin{enumerate}[(i)]
    \item\label{proofApproxCutInvSizeY} if~\(s \neq s^*\), then \(y_{i_{s}} \geq m -|B_{s}|\) and \(y_j < m - |B_{s}| \) for each child~\(j\) of~\(i_{s}\),
    \item\label{proofApproxCutInvSizeB} if~\(s \neq 0\), then \( \left(1-\frac{1}{2^s}\right)m < |B_s| \leq m \), and~\( 0 = |B_0| \leq m\),
    \item \label{proofApproxCutInvCutWidth} \(e_G(B_{s}, V\setminus B_{s} ) \leq s t \Delta(G) \), as well as
    \item\label{proofApproxCutInvBY} if~\(s \neq s^*\), then \(B_{s} \cap Y^{i_{s}} = \emptyset\).
  \end{enumerate}

  Clearly,~\ref{proofApproxCutInvSizeY}-\ref{proofApproxCutInvBY} hold for~\(s=0\). Fix an arbitrary~\(s \in [s^*]\) and assume that \ref{proofApproxCutInvSizeY}-\ref{proofApproxCutInvBY} hold for~\(s-1\) and consider the~\(s\)\th{}~execution of the outer while loop. If Line~\ref{proofApproxCutSetZ} of Algorithm~\ref{algoApproxCut} is executed in this iteration, let~\(Z_s\) be the set~\(Z\) computed there, and otherwise define~\(Z_s=\emptyset\). Let~\(i=i_{s-1}\) and define~\(k\), \(\ell\), and~\(j_h\) for~\(h \in [k]\)  as in the algorithm. Due to~\eqref{proofApproxCutEqYiPart} and~\ref{proofApproxCutInvBY} for~\(s-1\), the unions in Line~\ref{proofApproxCutUpdateB} and Line~\ref{proofApproxCutSetZ} are disjoint unions. Also, \ref{proofApproxCutInvSizeY} implies that Line~\ref{proofApproxCutSetZ} can be executed, if reached. Furthermore, \(|B_s| = |B_{s-1}| + |Z_s| + \sum_{h=1}^{\ell} \tilde{y}_{j_h} \). So, if Line~\ref{proofApproxCutSetZ} is executed, i.e., if~\(k=0\) or~\(\ell=k\), then~\(|B_s| = m\), \ref{proofApproxCutInvSizeB} is satisfied for~\(s\), and~\(s=s^*\). Otherwise,~\(k\geq 1\) and~\( \ell <k\). So, if~\(\tilde{y}_{j_1} > \frac{1}{2}(m-|B_{s-1}|)\), then \(|B_s|  >   \tfrac{1}{2}(m + |B_{s-1}|) \), as~\ref{proofApproxCutInvSizeY} ensures that~\(\ell \geq 1\). If~\(\tilde{y}_{j_1} \leq \frac{1}{2}(m-|B_{s-1}|)\), then \(\tilde{y}_{j_h} \leq \frac{1}{2} (m- |B_{s-1}|)\) for all~\(h \in [k]\) and \( |B_s| > \tfrac{1}{2}(m + |B_{s-1}|) \). Consequently, if Line~\ref{proofApproxCutSetZ} is not executed, then
  
  \[|B_s| \ > \ \tfrac{1}{2}(m + |B_{s-1}|) \ \geq \ \left(1 -\tfrac{1}{2^s} \right) m,\]  %the second one is \geq and not > because of the case s=1. 
  since~\ref{proofApproxCutInvSizeB} is satisfied for~\(s-1\). Hence,~\ref{proofApproxCutInvSizeB} is satisfied for~\(s\). As \(B_s \setminus B_{s-1} = Z_s \cup \left( \bigcup_{h\in [\ell]} \tilde{Y}^{j_h} \right)\), at most~\(e_G (i_{s-1})   \leq t \Delta(G)\) edges are cut when cutting off~\(B_s \setminus B_{s-1}\) from~\(G[V \setminus B_{s-1}]\) by \eqref{proofApproxCutEqYiCut} and therefore~\ref{proofApproxCutInvCutWidth} is satisfied for~\(s\).

  As argued earlier, if Line~\ref{proofApproxCutSetZ} is executed then~\(s= s^* \), and there is nothing to show for~\ref{proofApproxCutInvSizeY} and~\ref{proofApproxCutInvBY}. So assume~\(s\neq s^* \) from now on. This implies that~\(k \geq 1\), \(\ell < k\), \(Z_s = \emptyset\), and that Lines~\ref{proofApproxCutChildEll}-\ref{proofApproxCutGoDownEnd} are executed.  Note that Line~\ref{proofApproxCutChildEll} is feasible. The choice of~\(\ell\) in Line~\ref{proofApproxCutEll} implies that~\({y_{j_{\ell+1}} \geq \tilde{y}_{j_{\ell+1}} > m-|B_s|}\) when Line~\ref{proofApproxCutChildEll} is executed. Hence, the inner while loop is executed at least once during the~\(s\)\th{}~execution of the outer while loop and~\ref{proofApproxCutInvSizeY} is satisfied after the~\(s\)\th{}~execution of the outer while loop. Moreover, this implies that~\(i_{s}\) is a descendant of~\(j_{\ell +1}\) and hence \(Y^{i_s} \subseteq Y^{j_{\ell+1}} \subseteq \tilde{Y}^{j_{\ell+1}} \cup X^{i_{s-1}}\). So, to show that~\ref{proofApproxCutInvBY} is satisfied after the~\(s\)\th{}~execution of the outer while loop, it suffices to show that~\( ( \tilde{Y}^{j_{\ell+1}} \cup X^{i_{s-1}} ) \cap B_s = \emptyset\). Indeed,~\eqref{proofApproxCutEqYiPart}, the construction of~\(B_s\) from~\(B_{s-1}\), and~\(Z_s = \emptyset\) imply that
  
  \[ (\tilde{Y}^{j_{\ell+1}} \cup X^{i_{s-1}}) \cap B_{s} \ \subseteq \ (\tilde{Y}^{j_{\ell+1}} \cup X^{i_{s-1}}) \cap B_{s-1} \ \subseteq\ Y^{i_{s-1}} \cap B_{s-1} \ = \ \emptyset,\] 
  where the last equality holds because of~\ref{proofApproxCutInvBY} for~\(s-1\). This completes the proof of the invariants and shows that every step can be carried out. 

  The execution of the outer while loop stops as soon as~\(|B| > cm\). Therefore, the desired size requirements on~\(B\) are satisfied by~\ref{proofApproxCutInvSizeB}. Furthermore,~\ref{proofApproxCutInvSizeB} implies that \(s^* \leq \Bceil{\log ( 1/(1-c) ) }\). Consequently, the algorithm terminates, and the desired bound on the width of the output cut~\( (B,W)\) is satisfied by~\ref{proofApproxCutInvCutWidth}.
\end{proof}

\section{Proof of the Existence Part in Theorem~\ref{thmCutSpecSizes}}
%\section{Existence Part in of the Bisection for Graphs with a Given Tree Decomposition}
\label{secExistenceCut}

This section is devoted to proving the following theorem, which, for~\(m = \Bfloor{ n/2}\), clearly implies the existence part of Theorem~\ref{thmCutSpecSizes}. 

\begin{thm}[Existence Part of Theorem~\ref{thmCutSpecSizes}]
  \label{thmCutSpecSizesExistence}
  For every graph~\(G\) on~\(n\) vertices, every tree decomposition~\( (T,\X) \) of~\(G\) of width~\(t-1\), and for every~\(m \in [n]\), there is a cut~\((B,W)\) in~\(G\) with~\(|B| =m\) and
  
    \[ e_G(B,W) \ \leq \ \tfrac{ 1 }{2} t \Delta(G) \left( \left( \log  \tfrac{1}{r(T,\X)} \right)^2 \ + \  9 \log  \tfrac{1}{r(T,\X)} \ + \ 8 \right).\] 
  
\end{thm}

The first subsection introduces Theorem~\ref{thmDoubleR}, which is a generalized version of Lemma~\ref{lemmaDoubleDiam}, and uses it to prove Theorem~\ref{thmCutSpecSizesExistence}. The remaining subsections then concern the proof of Theorem~\ref{thmDoubleR}. In Section~\ref{subsecVertexLabeling}, we introduce some notation and a vertex labeling for the proof of Theorem~\ref{thmDoubleR} that will then be split into two cases. Section~\ref{subsecCase1} presents the short case and Section~\ref{subsecCase2} discusses the more involved case.

\subsection{Proof of Theorem~\ref{thmCutSpecSizesExistence}}
\label{subsecProofGenThm}

As explained and introduced in Section~\ref{subsecResultsGenGraphs}, when we are given a general graph~\(G\) and a tree decomposition~\( (T,\X)\) of~\(G\), we use the relative weight of a heaviest path~\(P\) in~\((T,\X)\). The next theorem is an extension of Lemma~\ref{lemmaDoubleDiam} and the heart of the proof of Theorem~\ref{thmCutSpecSizesExistence}. It uses the following notation: Consider a tree decomposition~\( (T,\X)\) with~\({\X = (X^i)_{i \in V(T)}}\) and a path~\(P \subseteq T\) with~\(P = (i_0, i_1, \ldots, i_{\ell})\). The end~\(i_0\) is called a \emph{nonredundant end} of~\(P\) if~\(X^{i_0} \neq \emptyset\) and, when~\(\ell \neq 0\), \(X^{i_h} \not\subseteq X^{i_{h-1}}\) for all~\(h \in [\ell]\). If one of the ends of~\(P\) is nonredundant, we say that~\(P\) is a \emph{nonredundant path}. Note that if~\((T,\X)\) is nonredundant, then any path~\(P \subseteq T\) will be nonredundant. 

\begin{thm}
  \label{thmDoubleR}
  For every graph~\(G\) on~\(n\) vertices, for every tree decomposition~\( (T, \X) \) of~\(G\) of width at most~\(t-1\), for every nonredundant path~\(P \subseteq T\), and for every~\(m \in [n]\),  there is a cut~\( (B,W,Z)\) in~\(G\) such that one of the following holds:
  \begin{enumerate}[1)]
    \item\label{thmDoubleROpt1} \(|B| = m\), \(Z=\emptyset\), and \(e_G (B,W) \leq 2 t \Delta(G) \), or
    \item\label{thmDoubleROpt2} \(|B| \leq m \leq |B| + |Z|\), \(0 < |Z| \leq n/2\), \(e_G(B,W,Z) \leq  t \Delta(G) \log ( 16/ w^*(P,\X) )\), and there is a tree decomposition~\((T', \X')\) of~\(G[Z]\) of width at most~\(t-1\) and a nonredundant path~\(P' \subseteq T'\) with~\(w^*(P', \X') \geq 2w^*(P,\X)\). 
  \end{enumerate}
\end{thm}

In Option~\ref{thmDoubleROpt1} in Theorem~\ref{thmDoubleR}, the bound is increased by a factor of~\(t \Delta(G)\) compared to Lemma~\ref{lemmaDoubleDiam} as we now use Lemma~\ref{lemmaRemoveClusterEdges} instead of cutting single edges. In Option~\ref{lemmaDoubleDiamOpt2} in Lemma~\ref{lemmaDoubleDiam}, we used cuts resulting from removing a vertex from the tree, which we extend by using Lemma~\ref{lemmaRemoveClusterEdges} and therefore obtain an extra factor of~\(t\) in the bound on~\(e_G(B,W,Z)\) in Option~\ref{thmDoubleROpt2} in Theorem~\ref{thmDoubleR}. Furthermore, in the proof of Option~\ref{thmDoubleROpt2}, we use the improved and more general version of the approximate cut, i.e., Lemma~\ref{lemmaApproxCut} instead of Lemma~\ref{lemmaApproxCutTree}, which improves the dependance on~\(r\).
Note that, if~\(P\) is a \emph{heaviest path} in~\((T,\X)\), i.e., a path of relative weight~\(r(T,\X)\), then Option~\ref{thmDoubleROpt2} implies that~\(r(T',\X') \geq 2r(T,\X)\), which is similar to Option~\ref{lemmaDoubleDiamOpt2} in Lemma~\ref{lemmaDoubleDiam}, where the relative diameter of a forest is doubled. Next, Theorem~\ref{thmCutSpecSizesExistence} is derived by applying Theorem~\ref{thmDoubleR} to the given graph and then, if the set~\(B\) does not contain the desired number of vertices, by repeatedly reapplying Theorem~\ref{thmDoubleR} to~\(G[Z]\) to cut off the remaining vertices. As the relative weight of a path is at most one, the process stops eventually with a set~\(B\) of the desired size.

\begin{proof}[Proof of Theorem~\ref{thmCutSpecSizesExistence}]
  Let~\(G=(V,E)\), \(n\), \(m\), \( (T,\X)\), and~\(t\) be as stated in the theorem. Without loss of generality, we may assume that~\( (T,\X)\) is nonredundant, because otherwise we can contract edges of~\(T\) while neither increasing the width of~\(T\) nor decreasing~\(r(T,\X)\). The following procedure describes how to find a cut~\( (B,W)\) in~\(G\) with~\({|B| =m}\) and width at most the bound stated in Theorem~\ref{thmCutSpecSizesExistence}. Fix a heaviest path~\(P\) in~\(T\) with respect to~\(\X\) and note that~\(P\) is nonredundant. The desired set~\(B\) is built iteratively by using Theorem~\ref{thmDoubleR}. To do so, we define~\(B_0 = \emptyset\), \(G_0 = G\), \( (T_0,\X_0) = (T,\X)\), \(P_0 = P\), and~\(m_s := m - |B_{s-1}|\) for all~\(s \geq 1\).  In the~\(s\)\th{}~step, we apply Theorem~\ref{thmDoubleR} to~\(G_{s-1}\) with the tree decomposition~\( (T_{s-1}, \X_{s-1})\), the nonredundant path~\(P_{s-1}\), and the parameter~\(m_s\), to partition the vertex set of~\(G_{s-1}\) into three sets~\(\tilde{B}_s\), \(\tilde{W}_s\), and~\(\tilde{Z}_s\). Then, we define~\(B_s :=  B_{s-1} \cupdot \tilde{B}_s\). If~\(|B_s| \neq m\), i.e., Case~2 occurs, let~\(G_s := G[\tilde{Z}_s] \) and denote by~\(P_s\) and~\( (T_s, \X_s)\) the nonredundant path~\(P'\) and the tree decomposition~\( (T',\X')\) from Theorem~\ref{thmDoubleR}. The construction stops when~\( |B_s| = m\), i.e., when the first index~\(s\) is reached such that~\(m_{s+1} = 0\). Note that when Case~1 occurs, then the construction stops as well. Furthermore, the construction will eventually stop because by Theorem~\ref{thmDoubleR} the size of the graph~\(G_s := G[ \tilde{Z}_s ]\) shrinks in every round. 
  
  Denote by~\(s^*\) the number of times we applied Theorem~\ref{thmDoubleR} to construct the desired cut. First of all, the final set~\(B_{s^*}\) will contain exactly~\(m\) vertices, because at most~\(m - |B_{s-1}|\) vertices are added to the set~\(B_{s-1}\) in the~\(s\)\th{}~iteration for all~\(s\in [s^*]\).  To state some invariants,  let~\(n_s\) be the number of vertices of~\(G_s\) and let~\(r_s := w^*(P_s, \X_s)\) for~\(s \in [s^*-1]\cup\{0\}\). For each~\(s \in [s^*]\), we have
  \begin{enumerate}[(i)]
    \item \(0 < m_s \leq |V(G_{s-1})|\) and \(B_{s-1} \cap V( G_{s-1} ) = \emptyset\),
    \item\label{proofCutSpecSizesInvR} if~\(s \neq s^*\) then~\(n_s \leq  n_{s-1} / 2\) and~\(r_s \geq 2 r_{s-1}\),
    \item\label{proofCutSpecSizesInvCutSize} \(e_{G_{s-1}} ( \tilde{B}_s, \tilde{W}_s, \tilde{Z}_s ) \leq t \Delta(G) \log ( 16 / r_{s-1})\).
  \end{enumerate}  
  Using that~\(r_{s-1} \leq 1\) for all~\(s \in [s^*]\), it is easy to check that all the above invariants are satisfied and that Theorem~\ref{thmDoubleR} can be applied in each step.  Let~\(r:= r_0 = r(T,\X)\). Now,~\ref{proofCutSpecSizesInvR} implies that~\({ r_s \geq 2^s r }\) for every~\(s \in [s^*-1]\cup\{0\} \), and hence 
  \begin{equation}
    \label{proofCutSpecSizesEqSStar}
    s^* \leq  \log (1/r ) +1.    
  \end{equation}
  Let~\(B:=B_{s^*}\) and \(W := V\setminus B\). With~\ref{proofCutSpecSizesInvCutSize}, we obtain the following for the width of~\( (B,W)\): 
   \begin{align*} 
    e_G (B,W) \ &\leq \ \sum_{s=1}^{s^*} e_{G_{s-1}} ( \tilde{B}_s, \tilde{W}_s, \tilde{Z}_s) \ \leq \ t \Delta(G) \sum_{s=1}^{s^*} \log \left( \frac{16}{2^{s-1} r} \right) \\
    & \leq \ t \Delta(G) \left( s^* \log \left( \tfrac{16}{r} \right) - \sum_{s=1}^{s^*} (s-1) \right) \ \leq \ t \Delta(G) \cdot s^* \left( \log \left( \tfrac{16}{r}\right) -\tfrac{1}{2} s^* + \tfrac{1}{2}  \right)  .
  \end{align*}  
  The last term is a quadratic equation in~\(s^*\) whose  maximum value is achieved at~\({s^* = \log (1/r ) + 9/2}\), which is larger than the upper bound given in~\eqref{proofCutSpecSizesEqSStar}. Therefore, the last term is increasing between~\(0\) and~\( \log(1/r) +1\) and the desired bound on~\(e_G(B,W)\) is obtained by replacing~\(s^*\) by~\(\log(1/r)+1\).
%   \begin{align*}
%    e_G (B,W) & \leq \ \tfrac{1}{2} t \Delta(G) \left( \left( \log  \tfrac{1}{r} \right)^2 + 9 \log  \tfrac{1}{r}  + 8 \right).
%  \end{align*}  
\end{proof}

\subsection{Further Notation and Vertex Labeling for the Proof of Theorem~\ref{thmDoubleR}}
\label{subsecVertexLabeling}

Here, we start with the proof of Theorem~\ref{thmDoubleR}. In this subsection and the remainder of Section~\ref{secExistenceCut}, let~\( G=(V,E)\) be an arbitrary graph on~\(n\) vertices and fix some integer~\(m \in [n]\). Furthermore, let~\( (T,\X)\) be a tree decomposition of~\(G\) of width at most~\(t-1\) with~\(T=(V_T, E_T)\) and~\( \X = (X^i)_{i \in V_T}\). Fix an arbitrary, nonredundant path~\(P = (V_P, E_P) \subseteq T\), denote by~\(i_0\) and~\(j_0\) the ends of~\(P\), and let~\(r:= w^*(P,\X)\). Without loss of generality assume that~\(i_0\) is a nonredundant end of~\(P\). 

First, we introduce some notation and a vertex labeling that depends on the path~\(P\), similar to the vertex labeling introduced in the proof of Lemma~\ref{lemmaTreeBase}.  For each node~\(i \in V_P\), let~\(T_i\) be the component of~\(T- E_P\) that contains~\(i\) and call~\(i\) the root of~\(T_i\). Moreover, define~\(R:= \bigcup_{i \in V_P} X^i\) and note that~\(|R| = rn\). For each~\(x\in R\), let the \emph{path node} of~\(x\) be the node~\(i \in V_P\) closest to~\(i_0\) with~\(x \in X^i\). Note that, as~\(P\) is a path and~\(i_0\) is one of its ends, such a node~\(i\) is unique. For every~\(i \in V_P\), define

\[ R_i := \left\{x \in X^i: i \text{ is the path node of } x \right\} .\] 
Furthermore, let~\(S:= V \setminus R\) and, for each node~\(i \in V_P\), let~\(S_i := \bigcup_{j \in V(T_i)} X^j \setminus R\). Note that the sets~\(R_i\) depend on the choice of~\(P\) and also the choice of~\(i_0\). The sets~\(R_i\) and the nodes on~\(P\) both correspond to the vertices in the path~\(P\) in the proof of Lemma~\ref{lemmaTreeBase2}: \(R_i\) is a subset of the vertices of~\(G\) and~\(V_P\) is a set of nodes of~\(T\). Similarly, the vertex sets~\(S_i\) and the node sets~\(V(T_i)\setminus\{i\}\) both correspond to the sets~\(T_v'\) in the proof of Lemma~\ref{lemmaTreeBase2}.

\begin{prop}\label{propPartitionRiSi}
  \begin{enumerate}[a)]
   \item[]
   \item \( \{ R_i : i \in V_P \} \cup \{ S_i: i \in V_P \} \)  is a partition of~\(V\).
   \item \(R_i \neq \emptyset\) for all~\(i \in V_P\).
  \end{enumerate}
\end{prop}

\begin{proof}
  \begin{enumerate}[a)]
    \item[]
    \item Clearly, \( \{R_i : i \in V_P\}\) is a partition of~\(R\) and~\( \bigcup_{i \in V_P} S_i = S\). Furthermore, property~(T3'), which can be found after Definition~\ref{defTreeDec}, implies that, for each~\(x \in S\), there is a unique~\(i \in V_P\) such that~\(x \in X^j\) only if~\(j\) in~\(T_i\). Hence, \(\{S_i : i \in V_P\}\) is a partition of~\(S\).
    \item  This follows easily by using~(T3') and that~\(P\) is nonredundant. \qedhere
  \end{enumerate}
\end{proof}

For each vertex~\(x\in S\), we say that~\(i \in V_P\) is the \emph{path node} of~\(x\) if~\(x \in S_{i}\). As the last proposition implies that~\(S_i\) and~\(S_{i'}\) are disjoint for distinct nodes~\(i,i'\) of~\(P\), every vertex~\(x \in V\) has a unique path node. When labeling the vertices of~\(G\) with~\(1,2, \ldots, n\), we say that the vertices in~\(Y\subseteq V\) receive \emph{consecutive labels} if there are~\(k, k' \in [n]\) and a bijection between the labels used in~\(Y\) and the set~\(\{k, k+1, \ldots, k'\}\).
A \emph{\(P\)-labeling}\label{Plabeling} of~\(G\) with respect to~\( (T,\X)\) is a labeling of the vertices in~\(V\) with~\( 1,2,\ldots, n\) so that
\begin{itemize}[leftmargin= \IdentationTDConditions]
  \item for each node~\(i \in V_P\), the vertices in~\(R_i \cup S_i\) receive consecutive labels and the vertices in~\(R_i\) receive the largest labels among those, and
  \item for all nodes~\(i,j \in V_P\) with~\(i \neq j\), if~\(i_0\) is closer to~\(i\) than to~\(j\), then each vertex in~\(R_i \cup S_i\) has a smaller label than every vertex in~\(R_j \cup S_j\). 
\end{itemize}
Clearly, such a~\mbox{\(P\)-labeling} always exists. From now on, fix a~\mbox{\(P\)-labeling} of the vertices in~\(G\) and consider any number that differs from a label in~\([n]\) by a multiple of~\(n\) to be the same as this label. When we talk about labels and vertices, in particular when comparing them, we always refer to the corresponding integer in~\([n]\).  For three vertices~\(a,b,c \in V\) with~\(a \neq c\), we say that~\(b\) \emph{is between~\(a\) and~\(c\)} if~\(b=a\), \(b=c\), or if increasing~\(a\) reaches~\(b\) before~\(c\). If~\(a=c\), then~\(b\) is between~\(a\) and~\(c\) if~\(b=a\). For example, when~\(n=10\), we say that~\(5\) is between~\(1\) and~\(7\), and~\(9\) is between~\(8\) and~\(3\).  As in the proof of Lemma~\ref{lemmaTreeBase} define~\(N_m(x) :=x + m\)  for every vertex~\(x \in V\).  Note that~\({N_m: V \rightarrow V}\) is a bijection and hence its inverse function~\(N_m ^{-1}\) is well defined.  For a set~\(Y \subseteq V\), define \(N_{m}(Y) := \{N_{m} (y): y \in Y\} \) and \(N_m^{-1} (Y) := \left\{ N_m^{-1} (y) : y \in Y \right\}\).

For a node~\(i \in V_P \setminus \{j_0\}\), we say that~\(j\) is the \emph{node after~\(i\) on~\(P\)} if~\(j \in V_P\) is the neighbor of~\(i\) that comes after~\(i\) when traversing~\(P\) from~\(i_0\) to~\(j_0\). Similarly, in this case we say that~\(i\) is the \emph{node before~\(j\) on~\(P\)}. Moreover, define~\(i_0\) to be the \emph{node after~\(j_0 \) on~\(P\)} and~\(j_0\) to be the \emph{node before~\(i_0\) on~\(P\)}. The next proposition shows how a \mbox{\(P\)-labeling} can be used to find cuts of small width in the corresponding graph. The analog in the case of a tree~\(\tilde{T}\), whose vertices are labeled along a path~\(\tilde{P}\) as in the proof of Lemma~\ref{lemmaTreeBase2}, would be that~\(\tilde{T}\) decomposes into disjoint parts when all edges incident to a vertex of~\(\tilde{P}\) are removed.

\begin{prop}
  \label{propCutPlabeling}
  Let~\(i\) be an arbitrary node in~\(P\), and denote by~\(i^{-}\) and~\(i^{+}\) the nodes before and after~\(i\) on~\(P\), respectively. Let~\(x^{-}\) be the vertex with the largest label in~\(R_{i^{-}}\) and let~\(x^{+}\) be the vertex with the smallest label in~\(S_{i^{+}} \cup R_{i^{+}}\). Moreover, if~\(i=i_0\) let~\(V_P^+ = V_P \setminus \{i_0\}\), if~\(i = j_0\) let~\(V_P^{-} = V_P\setminus\{j_0\}\), and otherwise let~\(V_P^-\) and~\(V_P^+\) be the node sets of the connected components of~\(P-i\), that contain~\(i^{-}\) and~\(i^{+}\), respectively. 
  Removing from~\(G\) the edges~\(E_G(i)\) decomposes~\(G\) into the following disjoint parts
  \begin{itemize}[leftmargin= \IdentationTDConditions]
      \item an isolated vertex for each~\(v \in R_i\),
      \item if~\(S_i \neq \emptyset\), the part~\(G[S_i]\),
      \item if~\( i \neq i_0\), the subgraph of~\(G\) induced by~\(\bigcup_{j \in V_P^-} (R_j \cup S_j) = \{1, \ldots, x^-\}\), and
      \item if~\( i \neq j_0\), the subgraph of~\(G\) induced by~\(\bigcup_{j \in V_P^+} (R_j \cup S_j) = \{x^+, \ldots, n\}\).
    \end{itemize}
\end{prop}

We omit the proof of this proposition, as it is a direct consequence of Lemma~\ref{lemmaRemoveClusterEdges}. Observe that it does not suffice to remove all edges that intersect with~\(R_i\) instead of the edges in~\(E_G(i)\). Note that Proposition~\ref{propCutPlabeling} also implies that,  for every vertex~\(v \in R\), the cut~\( (B',W')\) in~\(G\) with \(B' := \{1,2, \ldots, v\}\) and~\(W' := V \setminus G\) cuts at most~\( t\Delta(G)\) edges.

\subsection{Case~1 in the Proof of Theorem~\ref{thmDoubleR}} 
\label{subsecCase1}

Suppose that there is an~\(x \in R\) with~\(N_m(x) \in R\).

This case is similar to Case~1 in the proof of Lemma~\ref{lemmaTreeBase2}. Let~\(B\) be the set of vertices between~\({x+1}\) and~\(N_m(x)\), \(W := V \setminus B\), and~\(Z := \emptyset\), see also Figure~\ref{figCase1}. In this figure as well as in the following figures, we will draw the path~\(P\) in the tree~\(T\) horizontally on the top and under each node~\(i \in V_P\), the set~\(R_i\) will be visualized by a circle and the set~\(S_i\) will be visualized by a trapezoid. Clearly,~\({ |B| =m }\). Let~\(i\) and~\(j\) be the path nodes of~\(x\) and~\(N_m(x)\), respectively. Applying Proposition~\ref{propCutPlabeling} once with node~\(i\) and once with node~\(j\) shows that \(e_G(B,W) \leq e_G(i) + e_G(j) \leq 2t\Delta(G)\). 

\begin{figure}
  \begin{center}
    \includegraphics{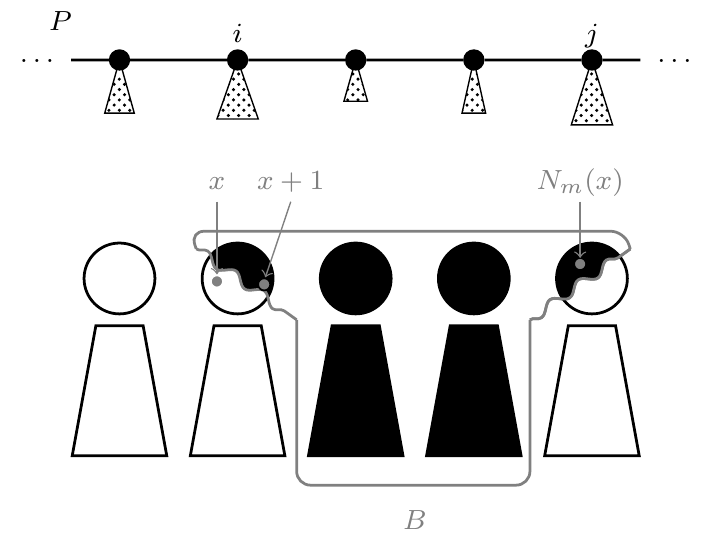}
  \end{center}
  \caption{The cut~\((B,W)\) in Case~1.  If a set is colored black, all its vertices are in the set~\(B\), if it is colored white, all its vertices are in the set~\(W\). The sets~\( R_i\) and~\(R_j\) can intersect~\(B\) as well as~\(W\).}
    \label{figCase1}
\end{figure}

\subsection{Case~2 in the Proof of Theorem~\ref{thmDoubleR}}
\label{subsecCase2}

Suppose that there is no~\(x \in R\) with~\(N_m(x) \in R\). 

Then,~\(N_m(x) \notin R\) and~\(N_m^{-1} (x) \notin R\) for every~\(x \in R\). Moreover, \( |R| = | N_m(R)| \leq |V \setminus R|\), which implies that~\(|R| \leq n/2 \) and hence 
 \begin{align} 
  \label{proofDoubleRBoundOnR}
  r \ = \ w^*(P,\X) \ =  \ \frac{|R|}{n} \ \leq \ \frac{1}{2}.
\end{align}

\subsubsection{Further Notation and Properties for Case~2}
\label{subsubsecCase2FurtherNotation}

If~\(i_0 \neq j_0\) and~\(i_0\) is not a neighbor of~\(j_0\) in~\(T\), then we denote by~\(T^+\) the graph obtained from~\(T\) by inserting the edge~\(\{i_0, j_0\}\). Otherwise we define~\(T^+ :=T\). The next proposition presents two observations that are easy to deduce using Proposition~\ref{propPartitionRiSi}.

\begin{prop}
  \label{propCaseIIPart1}
  In Case~2, the following statements hold:
  \begin{enumerate}[a)]
    \item For each~\(i\in V_P\), there is a node~\(j\in V_P\) such that \(N_m^{-1} (R_i) \subseteq S_j\) and \({N_m^{-1}(R_i) \cap S_{j'} = \emptyset}\) for every~\(j' \in V_P\) with~\(j' \neq j\). 
    \item For each~\(i\in V_P\), there is a node~\(j\in V_P\) such that \(N_m(R_i) \subseteq S_j\) and \({N_m(R_i) \cap S_{j'} = \emptyset}\) for every~\(j' \in V_P\) with~\(j' \neq j\).
  \end{enumerate}
\end{prop}

\begin{figure}
  \begin{center}
    \includegraphics{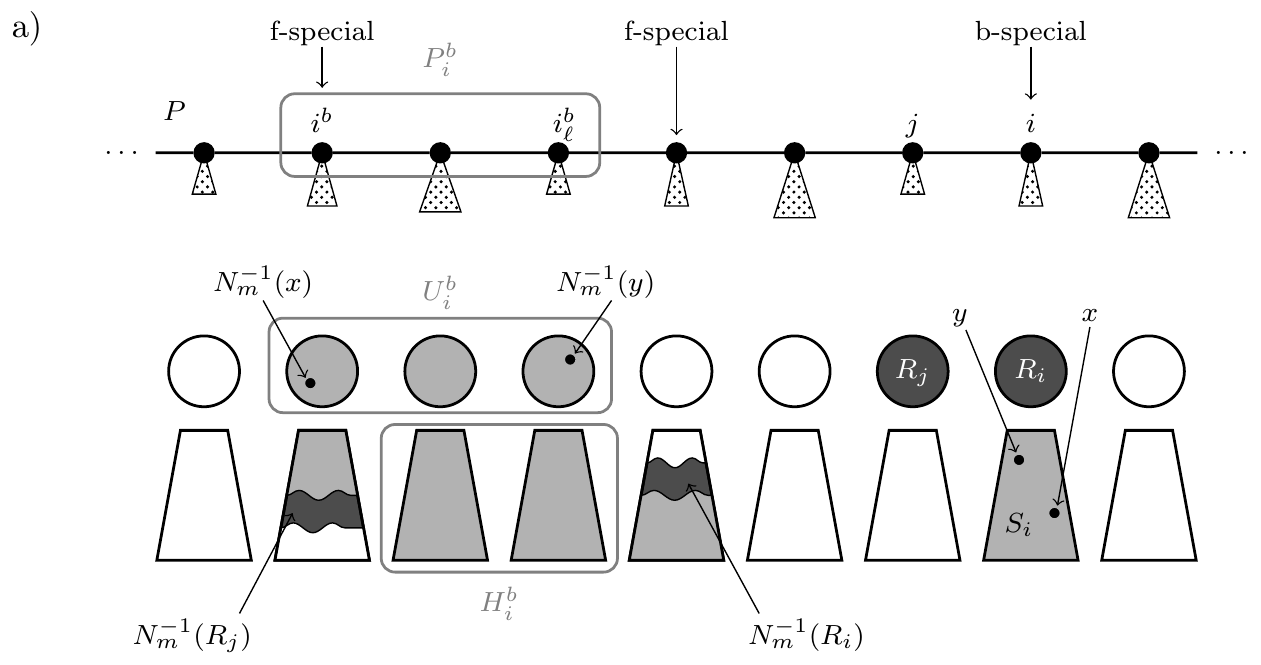}

    \vspace{1em}

    \includegraphics{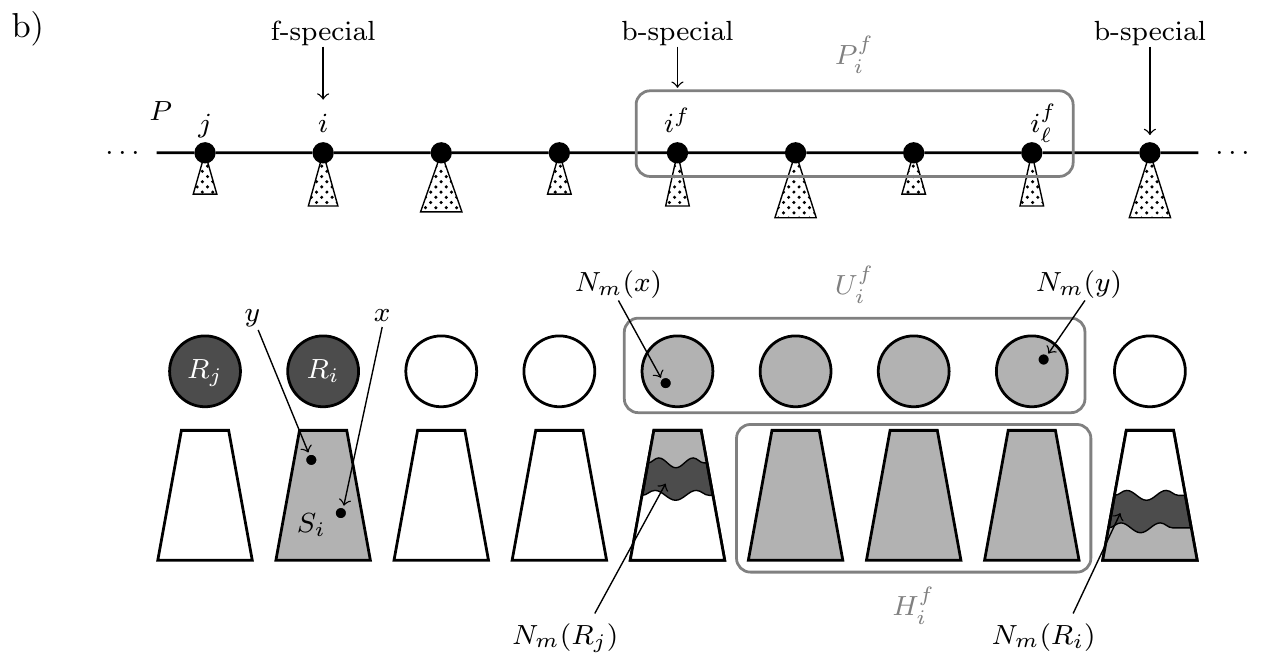}
  \end{center}
  \caption{Notation in Case~2. a)~A \mbox{b-special} node~\(i\in V_P\) and the sets~\(U_i^b\), \(P_i^b\), and~\(H_i^b\).  b)~An \mbox{f-special} node~\(i\in V_P\) and the sets~\(U_i^f\), \(P_i^f\), and~\(H_i^f\). In Part~a) and~b), the set~\(S_i\) and its image under~\(N_m^{-1}\) and~\(N_m\), respectively, are colored light gray. The sets~\(R_i\) and~\(R_j\) as well as their images under~\(N_m^{-1}\) and~\(N_m\), respectively, are colored dark gray.}
  \label{figFSpecialBSpecial}
\end{figure}

Next, the notion of special vertices from the proof of Lemma~\ref{lemmaTreeBase2} is extended to general graphs. A node~\(i\in V_P\) is called \emph{\mbox{b-special}} if the set~\(S_i\) contains a vertex~\(x\) with~\(N_m^{-1}(x) \in R\). A node~\(i\in V_P\) is called \emph{\mbox{f-special}} if~\(S_i\) contains a vertex~\(x\) with~\(N_m(x) \in R\).  For every~\(i\in V_P\), define
 \begin{align*}
  U_i^b :=  N_m^{-1} (S_i) \cap R \quad \quad \text{ and } \quad \quad  U_i^f :=  N_m(S_i) \cap R,
\end{align*}  
as well as
 \begin{align*}
  P_i^b := \left\{ j \in V_P : R_j \subseteq U_i^b \vphantom{U_i^f } \right\} \quad \quad \text{ and } \quad \quad  P_i^f := \left\{ j \in V_P : R_j \subseteq U_i^f \right\}.
\end{align*}  
Note that a node~\(i \in V_P\) is \mbox{b-special} if and only if~\(U_i^b \neq \emptyset\), and~\(i\in V_P\) is \mbox{f-special} if and only if~\(U_i^f \neq \emptyset\). See Figure~\ref{figFSpecialBSpecial} for a visualization of this and the next definitions. For each \mbox{b-special}~\(i\in V_P\), let~\(x\) be the smallest vertex in~\(S_i\) with~\(N_m^{-1}(x) \in R\) and let~\(i^b\) be the path node of~\(N_m^{-1} (x) \). Also, for each \mbox{b-special}~\(i\in V_P\), let~\(y\) be the largest vertex in~\(S_i\) with~\(N_m^{-1}(y) \in R\) and let~\(i_{\ell}^b\) be the path node of~\(N_m^{-1}(y)\). (We use~\(\ell\) as in large, as~\(R_{i_{\ell}^b}\) contains~\(N_m^{-1} (y)\).) Note that, if~\(i\) is \mbox{b-special}, then the nodes in~\(P_i^b\) induce a path or a cycle in~\(T^+\). If the nodes in~\(P_i^b\) induce a path in~\(T^+\), then the ends of this path are~\(i^b\) and~\(i^b_\ell\), which is similar to the tree case. Similarly, for each \mbox{f-special} node~\(i\in V_P\), let~\(x\) be the smallest vertex in~\(S_i\) with~\(N_m(x) \in R\) and let~\(i^f\) be the path node of~\(N_m(x)\). Also, let~\(y\) be the largest vertex in~\(S_i\) with~\(N_m(y) \in R\) and let~\(i_{\ell}^f\) be the path node of~\(N_m(y)\). Note that, if~\(i\) is \mbox{f-special}, then the nodes in~\(P_i^f\) induce a path or a cycle in~\(T^+\). If the nodes in~\(P_i^f\) induce a path in~\(T^+\), then the ends of this path are~\(i^f\) and~\(i^f_\ell\). Furthermore, define 
 \begin{align*}
  H_i^b = \bigcup_{j \in P_i^b \setminus \left\{i^b\right\}} S_j \quad \quad \text{ and } \quad \quad H_i^f = \bigcup_{j \in P_i^f \setminus \left\{i^f \right\}} S_j
\end{align*}  
for each b-special and f-special~\(i\in V_P\), respectively, and~\(H_i^b = \emptyset\) for each~\(i\in V_P\) that is not b-special as well as~\(H_i^f = \emptyset\) for each~\(i\in V_P\) that is not f-special. The next propositions state some properties about these sets as well as \mbox{b-special} and \mbox{f-special} nodes.

\begin{prop}
  \label{propCaseIIPart2}
  In Case~2, the following statements hold:
  \begin{enumerate}[a)]
    \item For each~\(i\in V_P\), there is a~\(j\in V_P\) such that~\(R_i \subseteq U_j^b\) and \(R_i \cap U_{j'}^b = \emptyset\) for every~\(j' \in V_P\) with~\(j' \neq j\). For each~\(i\in V_P\), there is a~\(j\in V_P \) such that~\(R_i \subseteq U_j^f\) and \( R_i \cap U_{j'}^f = \emptyset\) for every~\(j' \in V_P\) with~\(j' \neq j\). 
    \item \( \{ U_i^b : i \in V_P \} \) is a partition of~\(R\) and \( \{ U_i^f : i \in V_P \} \) is a partition of~\(R\).
    \item \( \{ P_i^b : i \in V_P, \text{ \(i\) is b-special} \} \) is a partition of~\(V_P\) and \\\( \{ P_i^f : i \in V_P, \text{ \(i\) is f-special} \} \) is a partition of~\( V_P\).
    \item For all~\(i\in V_P\), if~\(i \in P_i^b\) then~\(i = i^b\), and  if~\(i \in P_i^f\) then~\(i = i^f\).
    \item \(N_m (U_i^b \cup H_i^b) \subseteq S_i\) for all \mbox{b-special}~\(i\in V_P\) and \(N_m^{-1} (U_i^f \cup H_i^f) \subseteq S_i\) for all \mbox{f-special}~\({i\in V_P}\).
  \end{enumerate}
\end{prop}

\begin{proof}
  \begin{enumerate}[a)]
    \item[]
    \item Follows easily from Proposition~\ref{propCaseIIPart1}. 

    \item Part~a) implies that \(R = \bigcup_{i \in V_P} U_i^b = \bigcup_{i \in V_P} U_i^f\), while~\(R = \bigcupdot_{i \in V_P} R_i\) by Proposition~\ref{propPartitionRiSi}a). From these two, the statements follow.
    
    \item First, recall that Proposition~\ref{propPartitionRiSi}b) says that~\(R_i \neq \emptyset\) for all~\(i\in V_P\). Furthermore, Part~a) implies that~\(P_i^b \neq \emptyset\) if and only if~\(i\) is \mbox{b-special}, as well as~\(P_i^f \neq \emptyset\) if and only if~\(i\) is \mbox{f-special}. Now, the statement follows from Part~a) and Part~b).

    \item Let~\(i\) be an arbitrary node in~\(P\) satisfying~\(i \in P_i^b\) and recall that~\(R_j \neq \emptyset\) for all~\(j\in V_P\) by Proposition~\ref{propPartitionRiSi}b). Now,~\(i\) must be \mbox{b-special}, because otherwise~\(U_i^b = \emptyset\) and also~\(P_i^b = \emptyset\).  Moreover,~\(i \in P_i^b\) can only happen if~\(R_i \subseteq U_i^b\), i.e., \({R_i \subseteq N_m^{-1} (S_i)}\). Since~\(R_i \neq \emptyset\), the set~\(S_i\) contains a vertex~\(v\) such that~\({N_m^{-1}(v) \in R_i}\). Hence,~\(N_m^{-1} (w) \) is in~\(R_i \cup S_i\) for every~\(w \in S_i\), which is smaller than~\(v\). Consequently,~\(i^b = i\). The second part follows analogously.

    \item Assume that~\(i \in V_P\) is \mbox{b-special}. Note that every vertex in \(U_i^b \cup H_i^b \) is between~\(N_m ^{-1}(x)\) and~\(N_m^{-1} (y)\), where~\(x\) and~\(y\) are the smallest and the largest vertex in~\(S_i\) with~\(N_m^{-1}(x) \in R\) and~\(N_m^{-1}(y) \in R\), respectively. Therefore, the first inclusion is satisfied.  The second one follows similarly. \qedhere
  \end{enumerate}
\end{proof}

\begin{prop}
  \label{propBSpecialFSpecial}
  In Case~2, the following statements hold:
  \begin{enumerate}[a)]
    \item For every \mbox{b-special}~\(i\in V_P\), the node~\(i^b\) is \mbox{f-special}.
    \item For every \mbox{f-special}~\(i\in V_P\), the node~\(i^f\) is \mbox{b-special}. 
    \item A node~\(i\in V_P\) is \mbox{b-special} if and only if there is an \mbox{f-special} node~\(j\in V_P\) such that~\(i = j^f\).
    \item A node~\(i\in V_P\) is \mbox{f-special} if and only if there is a \mbox{b-special} node~\(j\in V_P\) such that~\(i = j^b\). 
  \end{enumerate}
\end{prop}

\begin{proof}
  \begin{enumerate}[a)]
    \item[]
    \item Let~\(i\in V_P\) be an arbitrary \mbox{b-special} node and denote by~\(j\)  the node before~\(i\) on~\(P\). By Proposition~\ref{propCaseIIPart1}a), there is a unique node~\(h\in V_P\) such that \(N_m^{-1} (R_j) \subseteq S_h\). As~\(R_j \neq \emptyset\) by Proposition~\ref{propPartitionRiSi}b), the node~\(h\) is \mbox{f-special}. As~\(i\) is \mbox{b-special},~\(h\) must be~\(i^b\).

    \item Similar to~a) by considering the node~\(j\) before~\(i\) on~\(P\) and applying Proposition~\ref{propCaseIIPart1}b) to~\(R_j\). 

    \item First, if there is an \mbox{f-special} node~\(j\in V_P\) such that~\(i = j^f\), then Part~b) shows that~\(i\) is \mbox{b-special}.  Furthermore,
     \begin{align*}%\label{proofDoubleRCountBFSp}
      & \left| \{i \in V_P: \text{\(i\) is b-special} \} \right| \  = \ \left| \{ i^b: i \in V_P, \text{ \(i\) is b-special} \} \right| \\
      \stackrel{a)}{\leq} \ &\left| \{ j \in V_P: \text{\(j\) is f-special} \} \right| \ 
      = \ \left| \{ j^f: j \in V_P, \text{ \(j\) is f-special} \} \right| \\
      \stackrel{b)}{\leq} \ & \left| \{ i \in V_P: \text{\(i\) is b-special} \} \right|, 
    \end{align*}  
    where the first equality holds because Proposition~\ref{propCaseIIPart2}c) implies that~\(i^b \neq j^b\) for two distinct \mbox{b-special} nodes~\(i, j\in V_P\) and the second equality holds analogously. Now, both inequalities must be equalities and the statement follows.

    \item Analog to Part~c). \qedhere

  \end{enumerate}
\end{proof}
  
\subsubsection{Accounting for Case~2}
\label{subsubsecCase2Accounting}

Proposition~\ref{propBSpecialFSpecial}c) and~d) imply
 \begin{align} 
  \label{proofDoubleRSumSi}
  \sum_{ \substack{j \in V_P: \\ j \text{ is \mbox{f-special}} }} \left| S_{j^f} \right| \ =  \sum_{ \substack{i \in V_P : \\ i \text{ is \mbox{b-special}} } } \left| S_i \right| \quad \quad \text{ and } \quad  \quad   \sum_{ \substack{j \in V_P: \\ j \text{ is \mbox{b-special}} }} \left| S_{j^b} \right| \ =  \sum_{ \substack{i \in V_P : \\ i \text{ is \mbox{f-special}} } } \left| S_i \right|.
\end{align}  
As~\(U_i^b\subseteq R\) and~\(H_i^b \subseteq S\) are disjoint, and similarly~\(U_i^f\) and~\(H_i^f\) are disjoint, Proposition~\ref{propCaseIIPart2}e) implies 
%  \begin{align} 
%   \label{proofDoubleRBInclusion}
%   \left| U_i^b \vphantom{U_i^f} \right| + \left| H_i^b \vphantom{U_i^f} \right| & \leq \left| S_i \vphantom{U_i^f}  \right| \quad \quad & & \text{ for every \mbox{b-special}~\(i\in V_P\) and } \\
%   \label{proofDoubleRFInclusion}
%   \left| U_i^f \right| + \left| H_i^f \right| & \leq \left| S_i\vphantom{H_i^f} \right| \quad \quad & & \text{ for every \mbox{f-special}~\(i\in V_P\).}
% \end{align}  
 \begin{align} 
  \label{proofDoubleRBFInclusion}
  \left| U_i^b \vphantom{U_i^f} \right| + \left| H_i^b \vphantom{U_i^f} \right|  \leq \left| S_i \vphantom{U_i^f}  \right| \quad \quad \text{ and } \quad \quad   \left| U_i^f \right| + \left| H_i^f \right|  \leq \left| S_i\vphantom{H_i^f} \right| \quad \quad  \text{ for every~\(i\in V_P\).}
\end{align}  
Recall that \(w(P, \X) = |R| = rn\) and hence~\(|S| = (1-r)n\). Proposition~\ref{propPartitionRiSi}a) and Proposition~\ref{propCaseIIPart2}c) imply that every vertex~\(v\in S\) is in exactly one set~\( H_i^b \cupdot S_{i^b} \) for some \mbox{b-special}~\(i\in V_P\) and in exactly one set~\(H_i^f \cupdot S_{i^f} \) for some \mbox{f-special}~\(i\in V_P\). Therefore,
 \begin{align} 
  \label{proofDoubleRBspecialPartition1}
  (1-r)n \ = \ |S| \ &= \ \sum_{\substack{i \in V_P: \\ i \text{ is \mbox{b-special}}}} \left( \left| S_{i^b}\vphantom{H_i^b} \right| + \left| H_i^b \right| \right) \quad \quad \text{ and } \\
  \label{proofDoubleRFspecialPartition1}
  (1-r)n \ = \ |S| \ &= \ \sum_{\substack{i \in V_P: \\ i \text{ is \mbox{f-special}}}} \left( \left| S_{i^f} \vphantom{H_i^f} \right| + \left| H_i^f \right| \right). \quad 
\end{align}  
Furthermore, Proposition~\ref{propCaseIIPart2}b) implies 
 \begin{align} 
  \label{proofDoubleRPartitionR}
  rn \ = \ |R| \ =  \sum_{ \substack{i \in V_P: \\ i \text{ is \mbox{b-special}} }} \left| U_i^b\right| \quad \quad\ \text{and} \quad \quad
  rn \ = \ |R| \ =  \sum_{ \substack{i \in V_P: \\ i \text{ is \mbox{f-special}} }} \left| U_i^f\right|.
\end{align}  
Now, Equations~\eqref{proofDoubleRBspecialPartition1}-\eqref{proofDoubleRPartitionR} imply
 \begin{align*}
  & \sum_{\substack{ i \in V_P: \\ i \text{ is \mbox{b-special}}}} \hspace{-0.5em} \left( \left| S_{i^b}\vphantom{H_i^f} \right| + \left| H_i^b \vphantom{H_i^f} \right| \right) \ + \sum_{\substack{ i \in V_P: \\ i \text{ is \mbox{f-special}}}} \hspace{-0.5em} \left( \left| S_{i^f} \vphantom{H_i^f} \right| + \left| H_i^f \right| \right) \ \ = \ \ 2 |S| \\
  = \ \ & \frac{1-r}{r} \ 2 |R| \ \ = \ \ \frac{1-r}{r} \left( \sum_{\substack{ i \in V_P: \\ i \text{ is \mbox{b-special}}}}  \left| U_i^b \vphantom{U_i^f} \right| \ + \sum_{\substack{ i \in V_P: \\ i \text{ is \mbox{f-special}}}}  \left| U_i^f \right|  \right) .
\end{align*}  
Now, we use~\eqref{proofDoubleRSumSi} to rearrange the terms in the sums on the left side, so that both sums use~\(|S_i|\) instead of~\( |S_{i^b}|\) and~\( |S_{i^f}|\), respectively, and obtain
 \begin{align*}
  & \sum_{\substack{ i \in V_P: \\ i \text{ is \mbox{b-special}}}} \hspace{-0.5em} \left( \left| S_{i} \vphantom{H_i^f}\right| + \left| H_i^b \vphantom{H_i^f} \right| \right) \ + \sum_{\substack{ i \in V_P: \\ i \text{ is \mbox{f-special}}}} \hspace{-0.5em} \left( \left| S_{i} \vphantom{H_i^f} \right| + \left| H_i^f \vphantom{H_i^f} \right| \right)  %\\ 
  =  \frac{1-r}{r} \ \left( \sum_{\substack{ i \in V_P: \\ i \text{ is \mbox{b-special}}}}  \left| U_i^b \vphantom{U_i^f} \right| \ + \sum_{\substack{ i \in V_P: \\ i \text{ is \mbox{f-special}}}}  \left| U_i^f \right| \right) .
\end{align*}  
As there is at least one \mbox{b-special} node~\(i \in V_P\) and at least one \mbox{f-special} node~\(i \in V_P\) in Case~2, the following proposition is obtained, which is similar to~\eqref{proofLemmaBase2_SpecialVertex} in the proof of Lemma~\ref{lemmaTreeBase2}.
\begin{prop}
  \label{propSpecialNodeForCase2ab}
  In Case~2, one of the following exists:
  \begin{enumerate}[a)]
    \item a \mbox{b-special} node~\(i\in V_P\) such that \( |S_i| + |H_i^b| \leq \left( \tfrac{1}{r} - 1 \right) | U_i^b| \), or
    \item an \mbox{f-special} node~\(i\in V_P\) such that \( |S_i| + |H_i^f| \leq \left( \tfrac{1}{r} -1 \right) | U_i^f| \).
  \end{enumerate}
\end{prop}
When constructing the cut in the next paragraphs, we treat the cases stated in the last proposition separately.

\subsubsection{Case~2a} 
\label{subsubsecCase2a}

There is a \mbox{b-special} node~\(i\in V_P\) such that \( |S_i| + |H_i^b| \leq \left( \tfrac{1}{r} - 1 \right) | U_i^b| \). 

Replacing~\(|S_i|\) with~\eqref{proofDoubleRBFInclusion}, we deduce that \(|U_i ^b| + 2 |H_i^b| \leq \left( \tfrac{1}{r} - 1 \right) |U_i^b| \), which implies
 \begin{align} 
  \label{proofDoubleRCase2aH3}
  \left|U_i^b\right| + \left|H_i^b\right| \ \leq \ \frac{1}{2r}  \left|U_i^b \right|. 
\end{align}  

Let \(Z := U_i^b \cupdot H_i^b \), which satisfies~\(Z \neq \emptyset\) as~\(i\) is \mbox{b-special} and hence~\(U_i^b \neq \emptyset\). Furthermore,~\(Z\) and~\(S_i\) are disjoint. Indeed, assume there is a vertex~\(x \in S_i \cap Z\). As~\(U_i^b \subseteq R\), we have~\(x \in S_i \cap H_i^b\). Due to Proposition~\ref{propPartitionRiSi}a) this implies that~\(i \in P_i^b\) and~\(i \neq i^b\), which contradicts Proposition~\ref{propCaseIIPart2}d). Hence,~\(S_i\) and~\(Z\) are disjoint and, with~\eqref{proofDoubleRBFInclusion}, it follows that~\(|Z| \leq n / 2\) as desired. For the tree decomposition~\( (T',\X')\) of~\(G[Z]\) required for Option~2 of Theorem~\ref{thmDoubleR}, we can use the restriction of~\( (T,\X)\) to~\(T' := T\) and~\(G[Z]\), which is indeed a tree decomposition of~\(G[Z]\) of width at most~\(t-1\).  Furthermore, define~\(P' := P\). Using that each vertex in~\(U_i^b\) is in~\(R \cap Z\) and hence contributes to the weight of~\(P'\) with respect to~\(\X'\), we obtain that~\(w(P', \X') \geq |U_i^b|\). Using that~\({|Z| \leq |U_i^b|/(2r)}\) by~\eqref{proofDoubleRCase2aH3}, it follows that~\(w^*(P',\X') \geq 2r\). Contract edges of~\(T'\) and~\(P'\) accordingly to turn~\((T',\X')\) into a nonredundant tree decomposition without increasing its width and note that~\(w^*(P',\X')\) is not decreased. Then~\(P'\) is a nonredundant path in~\((T',\X')\).

\begin{figure}
  \begin{center}
    \includegraphics{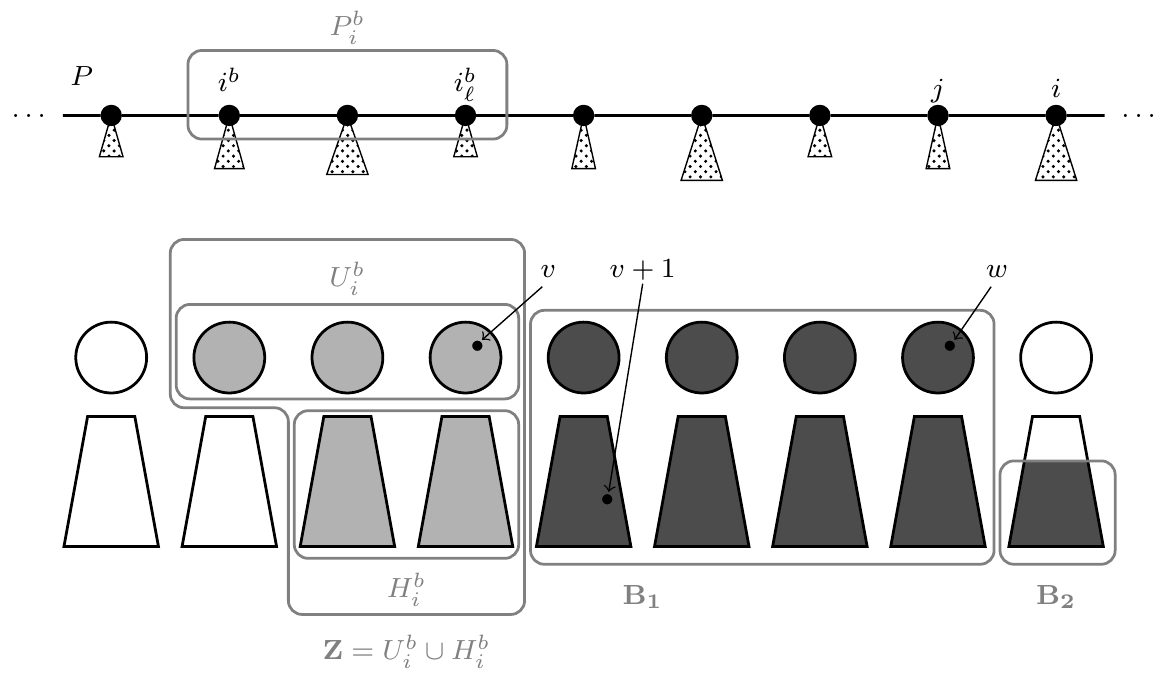}
  \end{center}
  \caption{Cut~\( (B,W,Z) \) with ~\(B=B_1 \cup B_2\) in Case~2a.}
  \label{figCase2a}
\end{figure}

Next, we will define the set~\(B\) for the cut~\( (B,W,Z)\) in~\(G\), see also Figure~\ref{figCase2a}. Let~\(j\) be the node before~\(i\) on~\(P\). Let~\(v\) be the largest vertex in~\(R_{i_\ell^b}\) and let~\(w\) be the largest vertex in~\(R_j\). If~\(v=w\), i.e.,~\(i^b_{\ell} = j\), define~\(B_1 := \emptyset\). Otherwise, define 

\[ B_1 := \{ u \in V: u \text{ is between } v+1 \text{ and } w\}.\] 
Let~\(\tilde{m} := m - |B_1|\). As~\(N_m(v)\) is in~\(S_i\), we know that \( |S_i| \geq \tilde{m} \geq 1\). Let~\((\tilde{T},\tilde{\X})\) be the restriction of~\((T,\X)\) to~\(\tilde{T}:=T\) and~\(G[S_i]\), and define \(c  := 1 - \frac{r}{1-r} = 2 -\frac{1}{1-r}\), which satisfies~\(0 \leq c < 1\) because of~\eqref{proofDoubleRBoundOnR}. If~\(c=0\), let~\(B_2 :=\emptyset\) and~\(W_2 := S_i\) and note that the cut~\( (B_2, W_2)\) in~\(G[S_i]\) satisfies~\(e_{G[S_i]} (B_2, W_2) = 0 \leq \log ( 2/r) \cdot t \Delta(G)\) and~\(c \tilde{m} \leq |B_2| \leq \tilde{m}\). Otherwise, apply Lemma~\ref{lemmaApproxCut} to~\(G[S_i]\) with the tree decomposition~\((\tilde{T},\tilde{\X})\), the size parameter~\(\tilde{m} = m - |B_1|\), and~\(c\)  to obtain a cut~\( (B_2, W_2)\) satisfying~\(B_2 \cupdot W_2 = S_i\), and \( c\tilde{m} \leq |B_2| \leq \tilde{m}\), as well as
 \begin{align} 
  \label{proofDoubleRCutB2}
  e_{G[S_i]} (B_2, W_2)  \  \leq  \ \Bceil{ \log \tfrac{1-r}{r} }t\Delta(G) 
  \ \leq \ \left( \log \tfrac{2(1-r)}{r} \right) t \Delta(G) \ \leq \ \log \left(\tfrac{2}{r}\right) \cdot t \Delta(G) .   
\end{align}  
Define~\(B := B_1 \cup B_2\). As~\(B_2 \subseteq S_i\), it contains no vertex from~\(B_1\), and therefore 
 \begin{align*}
  m - |B| \  = \ (m - |B_1|) - |B_2| \ \leq \ \tilde{m} - c \tilde{m} \ \leq \ (1-c) |S_i| \ = \ \tfrac{r}{1-r} |S_i| \ \leq \ | U_i ^b| \ \leq \ |Z|,
\end{align*}  
where the second to last inequality holds by the hypothesis of Case~2a. So, \( |B| \leq m \leq |B| + |Z|\) and the sets~\(B\) and~\(Z\) are disjoint by construction. Let~\( W := V \setminus ( B \cupdot Z)\). 

To finish Case~2a, we need to determine the width of the cut~\( (B,W,Z) \) in~\(G\). Applying Proposition~\ref{propCutPlabeling} with~\(i^b\), \(i^b _{\ell}\), and~\(i\) shows that

\[E_G (Z, B_1, S_i, W\setminus W_2) \ \subseteq \ E_G(i^b) \cup E_G(i^b_{\ell}) \cup E_G(i).\] 
Using that~\(B= B_1 \cupdot B_2\) and~\(S_i = B_2 \cupdot W_2\) as well as~\eqref{proofDoubleRCutB2}, we obtain

\[e_G( Z, B, W ) \ \leq \ 3t \Delta(G) + \log \left(\tfrac{2}{r}\right) \cdot t \Delta(G) \ \leq \ \log \left(\tfrac{16}{r}\right) \cdot t \Delta(G),\] 
as desired.

\subsubsection{Case~2b} 
\label{subsubsecCase2b}

There is an \mbox{f-special} node~\(i\in V_P\) such that \( |S_i| + |H_i^f| \leq \left( \tfrac{1}{r} - 1 \right) |U_i^f| \).

This case is similar to Case~2a but not completely analogous as we cannot simply reverse the labeling to obtain the situation of Case~2a, because this would violate the property that the vertices in~\(R_j\) receive the largest labels among all vertices in~\(R_j\cup S_j\) for every~\(j\in V_P\). Similarly to Case~2a,  we can replace~\(|S_i|\) in~\(|S_i| + |H_i^f| \leq \left(\tfrac{1}{r} - 1 \right)|U_i^f|\) with~\eqref{proofDoubleRBFInclusion} and derive~\( | U_i^f | + |H_i^f| \leq   | U_i^f | / (2r)\). Define \(Z := U_i^f \cupdot H_i^f\) and deduce analogously to Case~2a that~\(Z \neq \emptyset\) and~\( |Z| \leq n/2\) as~\(S_i\) and~\(Z\) are disjoint. As in Case~2a, let~\( (T',\X')\) be the restriction of~\( (T,\X)\) to~\(T' := T\) and~\(G[Z]\), define~\(P' := P\), and contract edges of~\(T'\) and~\(P'\) such that~\((T',\X')\) is nonredundant. Then~\(P'\) is nonredundant with respect to~\(\X'\) and~\(w^*(P',\X') \geq |U_i^f|/|Z| \geq 2r \). 

\begin{figure}
  \begin{center}
    \includegraphics{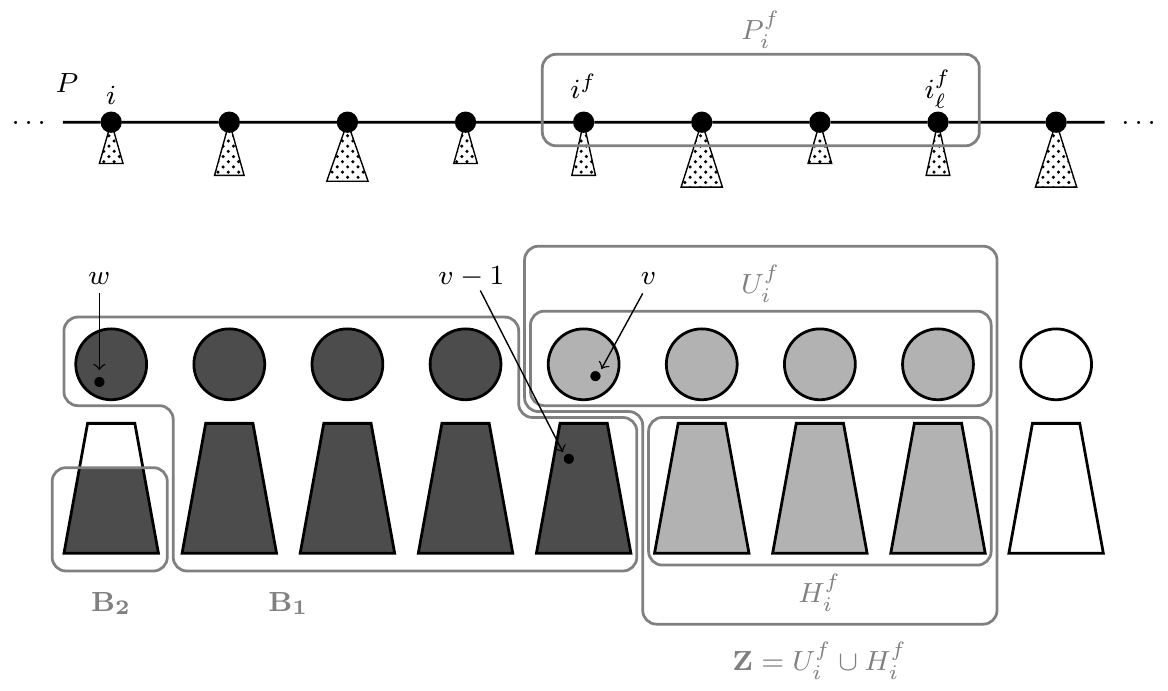}
  \end{center}
  \caption{Cut~\( (B,W,Z) \) with~\(B=B_1 \cup B_2\) in Case~2b.}
  \label{figCase2b}
\end{figure}

Let~\(w\) be the smallest vertex in~\(R_i\) and let~\(v\) be the smallest vertex in~\(R_{i^f}\). If~\(w = v\), i.e., if~\(i=i^f\), define~\({B_1 := \emptyset}\). Otherwise, define \( B_1 := \{u \in V : \text{\(u\) is between \(w\) and \(v-1\)}\}\), see also Figure~\ref{figCase2b}. Let~\(\tilde{m} := m- |B_1|\). As~\(N_m^{-1} (v)\) is in~\(S_i\), we have \(|S_i| \geq \tilde{m} \geq 1\). As in Case~2a, let~\(c := 1- \frac{r}{1-r}\) and if~\(c=0\), define~\(B_2 := \emptyset\) and~\(W_2:=S_i\). Otherwise we apply Lemma~\ref{lemmaApproxCut} to~\(G[S_i]\) with parameters~\( \tilde{m}\) and~\(c\), to obtain a cut~\( (B_2,W_2) \) with~\(B_2 \cupdot W_2 = S_i\) and~\(c \tilde{m} \leq |B_2| \leq \tilde{m}\). Defining \(B = B_1 \cup B_2\), one can argue that \( |B| \leq m \leq |B| + |Z|\) and that~\(B \cap Z = \emptyset\). Let~\(W:= V\setminus (B \cup Z)\). Similarly to Case~2a, Proposition~\ref{propCutPlabeling} implies that \(E_G(S_i, B_1, Z, W\setminus W_2) \subseteq E_G(i) \cup E_G(i^f) \cup E_G(i^f_{\ell})\), which together with the bound on~\(e_{G[S_i]} (B_2, W_2)\) from Lemma~\ref{lemmaApproxCut} gives the desired bound on the cut width and completes the proof of Theorem~\ref{thmDoubleR}.

\section{Algorithms for Graphs with a Given Tree Decomposition}
\label{secAlgoDetails}

In this section, we present two algorithms that, when given a tree decomposition of a graph, compute a bisection within the bound stated in Theorem~\ref{thmCutSpecSizes}. Section~\ref{subsecImplPrelim} discusses the preliminaries for both algorithms and gives an overview on subroutines which are described in Appendix~\ref{subsecAppSubroutines}. Section~\ref{subsecImplAlgo1} and Section~\ref{subsecImplAlgo2} describe both implementations, where the second one runs in linear time and is based on the first one, which is easier. The first implementation runs only for certain tree decompositions in linear time. All algorithms are described in a concise way and the reader is referred to~\cite{ThesisTina} for a more detailed description.

\subsection{Preliminaries for Both Implementations}
\label{subsecImplPrelim} 
%\label{subsecAlgoAspects}

Note that the algorithm in Theorem~\ref{thmCutSpecSizes} receives a tree decomposition as input. In general, it is NP-hard to compute a tree decomposition of minimum width~\cite{ArnborgCorneil}, but for fixed~\(k\in \mathbb{N}\) there is an algorithm that, when given a graph~\(G\) with~\({\tw(G) \leq k}\), computes a tree decomposition of width at most~\(k\) in linear time \cite{Bodlaender96}. For the implementation we always assume that the input graph~\(G\) satisfies~\(V(G)=[n]\) for some integer~\(n\), and that the clusters of the provided tree decomposition~\((T,\X)\) are given as unordered lists. Moreover, we assume that~\(T\) is given by its adjacency lists and that each node of~\(T\) has a link pointing to its cluster. In all implementations, we only use~\((T,\X)\) and not the graph~\(G\) itself. However, when explaining the ideas of the algorithm, we might still refer to the underlying graph~\(G\). Sets, and in particular the set~\(B\) of the desired bisection, are stored as unordered lists of vertices of~\(G\), unless indicated otherwise. Therefore the union of two disjoint sets is a simple concatenation of lists and takes constant time. Furthermore, our algorithms that compute a cut~\((B,W)\) only output a list of the vertices in~\(B\), as a list of the vertices in~\(W\) can then be computed in~\(\bigO(n)\) time using the assumption that the vertex set of~\(G\) is~\([n]\). Table~\ref{tableRunninTimes} gives an overview on the subroutines, which are used by both implementations. 

\begin{table}[b]
\begin{tabular}{l@{\hspace{1em}}c@{\hspace{1em}}c@{\hspace{1em}}l}
  \toprule
  Algorithm/Task & Running Time & Statement & Implementation \\
  \midrule
  Approximate cut~\( (B,W)\) & \( \bigO( \| (T,\X) \| ) \) & Lemma~\ref{lemmaApproxCut} & Appendix~\ref{subsubsecImplApproxCut} \\ 
  Restriction of~\( (T,\X)\) to~\(T'\) and~\(G'\) & \( \bigO(\| (T,\X) \|) \) & \multicolumn{2}{c}{mentioned in Section~\ref{subsecImplPrelim}} \\
  Make~\((T,\X)\) nonredundant & \( \bigO( \|(T,\X)\| ) \) & Proposition~\ref{propNonRedundantTD}b) &  Appendix~\ref{subsubsecImplNonRedundantTD}\\ 
  Heaviest path in~\( (T,\X)\) & \( \bigO( \| (T,\X) \| ) \) & Lemma~\ref{lemmaHeaviestPath} & Appendix~\ref{subsubsecImplHeaviestPath} \\
  \mbox{\(P\)-labeling} for a path~\( P \subseteq T\) & \( \bigO( \| (T,\X) \| ) \) & Lemma~\ref{lemmaPLabeling} & Appendix~\ref{subsubsecImplLabeling} \\
  \bottomrule
\end{tabular}
\caption{Overview on subroutines. Given a tree decomposition~\( (T,\X)\) of an arbitrary graph with vertex set~\([n]\) for some integer~\(n\), we describe subroutines for some tasks.}
\label{tableRunninTimes}
\end{table}

The algorithmic part of Theorem~\ref{thmCutSpecSizes} is more involved than the algorithmic part of Theorem~\ref{thmTreeCutSpecSizes} as we need to deal with the tree decomposition in an efficient way.  First of all, only a tree decomposition~\( (T_0, \X_0)\) of the input graph~\(G_0\) is provided and, to reapply Theorem~\ref{thmDoubleR} to a subgraph~\(G\) of~\(G_0\), a restriction~\((T,\X)\) of~\((T_0,\X_0)\) to some suitable tree~\(T\) and~\(G\) needs to be computed. The restriction~\( (T,\X)\) can be computed in~\( \bigO( \| (T_0, \X_0) \| ) \) time, provided that we can determine in~\(\bigO(1)\) time whether a vertex~\(x \in V(G_0)\) is in~\(G\).  Observe that~\(\|(T,\X)\| \leq \| (T_0, \X_0)\|\), but we have no control over the shrinkage of the size of the tree decomposition, even if we know that~\(G\) has at most half of the vertices of~\(G_0\). The next proposition about nonredundant tree decompositions will help us to control the size of the tree decomposition when reapplying Theorem~\ref{thmDoubleR} to construct a bisection as in Section~\ref{subsecProofGenThm}.

\begin{prop}
  \label{propNonRedundantTD}
  For every graph~\(G\) and every tree decomposition~\((T,\X)\) of~\(G\), the following holds.
  \begin{enumerate}
    \item[a)] If~\((T,\X)\) is nonredundant, then~\(|V(T)| \leq |V(G)|\).
    \item[b)] In~\(\bigO\left(\|(T, \X )\|\right)\) time, one can transform~\((T,\X)\) into a nonredundant tree decomposition~\((T',\X')\) such that \( \|(T',\X')\| \leq \| (T,\X)\| \), \(r(T', \X') \geq r(T,\X)\), and the width of~\( (T',\X')\) is at most the width of~\( (T,\X)\). 
  \end{enumerate}
\end{prop}

A proof for Part~a) of the last proposition can be found in~\cite{KleinbergTardos} (see Fact~10.16). The proof of Part~b), except for the running time, which is discussed in Appendix~\ref{subsubsecImplNonRedundantTD}, is straightforward when contracting edges of~\(T\). Moreover, an algorithm following the construction in Section~\ref{secExistenceCut} needs to compute a heaviest path~\(P\) and a corresponding labeling of the vertices, as stated in the next lemmas, which are proved in Section~\ref{subsubsecImplHeaviestPath} and Section~\ref{subsubsecImplLabeling}.

\begin{lemma}\label{lemmaHeaviestPath}
  For every tree decomposition~\((T,\X)\), a heaviest path in~\(T\) with respect to~\(\X\) can be computed in~\(\bigO(\|(T,\X)\|)\) time.  
\end{lemma}

%Kleinberg-Tardos Book: 
%10.16: Any nonredundant tree decomposition of an n-node graph has at most n pieces
%10.13: Suppost that T-t has components T_1 ... T_d. Then the subgraphs \(G_{T_1} - V_t, G_{T_2} - V_t, ... G_{T_t}-V_t\) have no nodes in common, and there are no edges between them.

\begin{lemma}\label{lemmaPLabeling}
  Given a tree decomposition~\((T,\X)\) of a graph~\(G=(V,E)\) and a path~\(P \subseteq T\), one can compute a \mbox{\(P\)-labeling} of~\(G\) with respect to~\((T,\X)\) in~\(\bigO(\|(T,\X)\|)\) time. While doing so, we can compute the following parameters (using the same notation as in Section~\ref{subsecVertexLabeling}):
  \begin{itemize}[leftmargin= \IdentationTDConditions]
    \item two integer arrays~\(A_L\) and~\(A_V\), each of length~\(n\), so that for~\(x \in V\) the entry~\(A_L[x]\) is the label of vertex~\(x\) and for~\(k \in [n]\) the entry~\(A_V[k]\) is the vertex that received label~\(k\), 
    \item a binary array~\(A_R\) of length~\(n\), so that for~\(x \in V\) the entry~\(A_R[x]\) is one if and only if~\(x \in R\), 
    \item an integer array~\(A_P\) of length~\(n\), so that for~\(x\in V\) the entry~\(A_P[x]\) is the path node of~\(x\), and
    \item a list~\(L_P\) of the nodes on the path~\(P\) in the order in which they occur when traversing~\(P\), including, for each~\(i \in V(P)\), a pointer to the root of~\(T_i\) stored as an arborescence with root~\(i\).
  \end{itemize}
\end{lemma}

Note that the arrays~\(A_L\) and~\(A_V\) allow us to convert vertex names of the input graph to labels and vice versa in constant time. Therefore, in the following, we do not distinguish between vertex names and labels, except when we modify the \mbox{\(P\)-labeling} in Section~\ref{subsecImplAlgo2}.

\subsection{The First Implementation}
\label{subsecImplAlgo1}
%%% The O(nt) implementation

Consider a tree decomposition~\( (T,\X)\) of width~\(t-1\) of an arbitrary graph~\(G\) on~\(n\) vertices. The aim of this subsection is to describe an algorithm that computes a bisection in~\(G\) of width within the bound stated in Theorem~\ref{thmCutSpecSizes} in~\( \bigO( \| (T,\X) \| + nt) \) time. First note that due to Proposition~\ref{propNonRedundantTD} we may assume that~\( (T,\X)\) is nonredundant and then the running time simplifies to~\(\bigO(nt)\). Assume that we can compute a cut~\( (B,W,Z)\) with the properties of Theorem~\ref{thmDoubleR} in~\( \bigO( \| (T,\X) \|)\) time when given~\( (T,\X)\), a path~\(P \subseteq T\), and an integer~\(m \in [n]\). Then it is easy to compute in~\(\bigO(nt)\) time a bisection in~\(G\) with the properties in Theorem~\ref{thmCutSpecSizes} by following the construction from Section~\ref{subsecProofGenThm}. Indeed, we can compute a heaviest path~\(P\) in~\( (T,\X)\), apply Theorem~\ref{thmDoubleR} with the path~\(P\) and~\(m = \Bfloor{n/2}\) to obtain a cut~\((B,W,Z)\) in~\(G\). If~\(|B| < m\), i.e., Option~\ref{thmDoubleROpt2} in Theorem~\ref{thmDoubleR} occurred, we compute the restriction of~\( (T,\X)\) to~\(T\) and~\(G[Z]\) to obtain a tree decomposition~\( (T',\X')\) of~\(G[Z]\), which we modify to be nonredundant. Then the same procedure can be applied to~\((T',\X')\) with parameter~\(m' = m-|B|\). This is repeated until all sets~\(B\) together contain exactly~\(m\) vertices. In one iteration, each of these steps takes~\(\bigO(\|(T,\X)\|)\) time, which is~\(\bigO(nt)\) time by Proposition~\ref{propNonRedundantTD}a). Note that Theorem~\ref{thmDoubleR} requires a nonredundant path~\(P\) and that this is satisfied because every path in a nonredundant tree decomposition is nonredundant. As Theorem~\ref{thmDoubleR} guarantees that the number of vertices shrinks by a factor of~\(1/2\) from~\(G\) to~\(G[Z]\), this implies that a bisection of width within the bound of Theorem~\ref{thmCutSpecSizes} can be computed in~\(\bigO(nt)\) time when a nonredundant tree decomposition of~\(G\) is provided as input. 

So for the first implementation, it only remains to show how to implement the algorithm corresponding to Theorem~\ref{thmDoubleR}. Let~\(G=(V,E)\) be a graph on~\(n\) vertices, let~\((T,\X)\) be a nonredundant tree decomposition of~\(G\) of width~\(t-1\), and let~\(P \subseteq T\) be a path in~\((T,\X)\). Fix an~\(m \in [n]\) and define~\(V_T := V(T)\), \(V_P := V(P)\), \(\X := (X^i)_{i \in V_T}\), and~\(r:= w^*(P,\X)\). Recall the construction from Sections~\mbox{\ref{subsecVertexLabeling}-\ref{subsecCase2}} and in particular the labeling as well as the sets~\(R:= \bigcup_{i \in V_P} X^i\) and~\(S := V\setminus R\).  First, the algorithm computes a \mbox{\(P\)-labeling} of the vertices of~\(G\) as well as the arrays~\(A_L\), \(A_V\), \(A_R\), and~\(A_P\) from Lemma~\ref{lemmaPLabeling}, which takes~\( \bigO( \|(T,\X)\|)\) time by Lemma~\ref{lemmaPLabeling}.  Using~\(A_R\), the algorithm can decide in~\(\bigO(n)\) time whether Case~1 of the proof of Theorem~\ref{thmDoubleR} applies, i.e., whether there is a vertex~\(x\in R\) with~\(N_m(x) \in R\). If Case~1 applies, then~\(Z:=\emptyset\) and a list of the vertices in the set~\(B\) can be obtained from~\(A_V\) in~\(\bigO(n)\) time.

So from now on assume that Case~1 does not apply. Recall the definitions of the sets~\(S_i\), \(U_i^b\), \(U_i^f\), \(H_i^b\), and~\(H_i^f\) for~\(i\in V_P\), as well as the notion of b-special and f-special nodes from Section~\ref{subsubsecCase2FurtherNotation}. First, the sets~\(S_i\), \(U_i^b\), \(U_i^f\), \(H_i^b\), and~\(H_i^f\) are indexed by nodes of~\(P\) and, as~\(|V_P| \leq |V_T| \leq n\) by Proposition~\ref{propNonRedundantTD}a), there are at most~\(n\) of each such sets.  Since, for a fixed~\(i \in V_P\), \(x \in S_i\) if and only if the path node of~\(x\) is~\(i\) and~\(x \not\in R\), the sizes of all sets~\(S_i\) can be computed in~\(\bigO(n)\) time by using the arrays~\(A_R\) and~\(A_P\). Furthermore, consider an arbitrary~\(x \in V\) and some node~\(i \in V_P\). Then,~\(x \in U_i^b\) if and only if~\(N_m(x) \in S_i\) and~\(x \in R\). Hence, the sizes of all sets~\(U_i^b\) can be computed in~\(\bigO(n)\) time. While doing so, for each~\(i \in V_P\), the algorithm keeps track of the smallest vertex~\(x \in S_i\) with~\(N_m^{-1}(x) \in R\) and the largest vertex~\(y \in S_i\) with~\(N_m^{-1}(y) \in R\), if there are any. If such~\(x\) and~\(y\) exist, then~\(H_i^b\) is the set of vertices between~\(N_m^{-1}(x)\) and~\(N_m^{-1}(y)\), which are not in~\(U_i^b\), and if such~\(x\) and~\(y\) do not exist then~\(H_i^b = \emptyset\) as~\(i\) is not \mbox{b-special} in this case. Similarly, the sizes of the sets~\(U_i^f\) and~\(H_i^f\) for~\(i\in V_P\) can all together be computed in~\(\bigO(n)\) time. Now, the algorithm checks for each node~\(i \in V_P\) whether it satisfies one of the properties stated in Proposition~\ref{propSpecialNodeForCase2ab}. To do so, it uses that a node~\(i \in V_P\) is \mbox{b-special} if and only if~\(U_i^b \neq \emptyset\) and similarly for \mbox{f-special}. Hence, such a node~\(i \in V_P\) can be found in~\(\bigO(n)\) time. 

Assume the algorithm determined a node~\(i \in V_P\) with the property of Proposition~\ref{propSpecialNodeForCase2ab}a). So the algorithm follows the construction described in Section~\ref{subsubsecCase2a}. The case when~\(i\) satisfies Proposition~\ref{propSpecialNodeForCase2ab}b) is analogous. The set~\(B_1\) can be read off the labeling in~\(\bigO(n)\) time. To obtain the set~\(B_2\), an approximate cut with~\(\tilde{m} := m - |B_1|\) and~\(c := 1- \frac{r}{1-r}\) is constructed in the subgraph of~\(G[S_i]\). In order to do so, a tree decomposition of~\(G[S_i]\), where the vertices of~\(G[S_i]\) are renamed to~\(1, \ldots, |S_i|\), is needed. The algorithm uses the restriction~\((\tilde{T}, \tilde{\X})\) of~\( (T,\X)\) to~\(T\) and~\(G[S_i]\), which can be obtained in~\(\bigO( \| (T,\X) \|)\) time. While computing~\((\tilde{T}, \tilde{\X})\), each vertex name is converted to its label and~\(A_L[u]-1\) is subtracted, where~\(u\) is the vertex with the smallest label among all vertices in~\(S_i\). Then, the algorithm in Lemma~\ref{lemmaApproxCut} takes~\(\bigO( \| (\tilde{T}, \tilde{\X}) \| ) = \bigO( \| (T,\X)\| ) \) time to compute~\(B_2\). Afterwards, the computation of~\(B:=B_1 \cupdot B_2\) takes~\(\bigO(n)\) time as the vertex names in~\(B_2\) need to be converted back.

Next, the set~\(Z:= U_i^b \cup H_i^b\) is computed. To do so, note that~\(Z\) is the set of vertices between~\(N_m^{-1}(x)\) and~\(N_m^{-1}(y)\), where~\(x\) and~\(y\) are the smallest and the largest vertices in~\(S_i\) with~\(N_m^{-1}(x) \in R\) and~\(N_m^{-1}(y) \in R\), respectively, and that~\(x\) and~\(y\) have been computed earlier. Hence, a list of the vertices in~\(Z\) can be read off the labeling. Similar to the preparation for the application of Lemma~\ref{lemmaApproxCut} to~\(G[S_i]\), the algorithm sets up a bijection, or it adjusts an existing bijection, so that the vertices of~\(G[Z]\) are~\(1, \ldots, |Z|\) after applying the bijection. This takes~\(\bigO(n)\) time. When discussing the implementation of the subroutines, e.g., computing an approximate cut in Appendix~\ref{subsubsecImplApproxCut} and computing a heaviest path in Appendix~\ref{subsubsecImplHeaviestPath}, it will become clear that these subroutines require the underlying graph~\(\tilde{G}\) to satisfy~\(V(\tilde{G}) = [\tilde{n}]\) for some integer~\(\tilde{n}\). Using the same arrays in all recursive calls with a single initialization in the beginning, one can avoid the relabeling while keeping the same running time.

All in all, each step of the procedure takes~\(\bigO(n)\) time, except the computation of the \mbox{\(P\)-labeling} as well as the application of Lemma~\ref{lemmaApproxCut}, that each takes~\(\bigO( \| (T,\X) \| )\) time. Using that~\((T,\X)\) is nonredundant, the desired running time follows.

\subsection{The Second Implementation}
\label{subsecImplAlgo2}

Consider a tree decomposition~\((T,\X)\) of width~\(t-1\) of a graph~\(G\) on~\(n\) vertices. Note that the running time~\(\bigO(\|(T,\X)\| + nt)\) of the first implementation in Section~\ref{subsecImplAlgo1} is not necessarily linear in the input size, e.g.,~if~\(t\) is not constant and~\( (T,\X)\) has only few clusters of size~\(\theta(t)\). Here we describe an algorithm that computes a bisection of width within the bound of Theorem~\ref{thmCutSpecSizes} in~\(\bigO(\|(T,\X)\|)\) time. The first implementation might not run in~\(\bigO(\|(T,\X)\|)\)  time for the following reason: Computing a heaviest path and making a tree decomposition nonredundant both require to go through the entire current tree decomposition, whose size might not shrink in the same ratio as the number of vertices of the current graph does, and therefore each take~\(\Theta(\|(T,\X)\|)\) time. So, in order to compute a bisection in~\(\bigO(\|(T,\X)\|)\) time, we cannot do this before every application of Theorem~\ref{thmDoubleR}. Therefore, in the second implementation, we make the tree decomposition nonredundant and compute a heaviest path~\(P\) only once in the beginning. For the second and later applications of Theorem~\ref{thmDoubleR}, we reuse a piece of~\(P\) and adjust the arrays of Lemma~\ref{lemmaPLabeling} so that a tree decomposition is only needed for the application of the approximate cut from Lemma~\ref{lemmaApproxCut}. We apply Lemma~\ref{lemmaApproxCut} in such a way that we only need a small part of~\(T\), and this part of the tree decomposition can be obtained by computing a restriction of the initial input tree decomposition. All in all, we do not need to carry along a tree decomposition of the current graph and, in all applications of Lemma~\ref{lemmaApproxCut} together, we only traverse once the initial tree decomposition.  This is made precise with the next lemma, which is an algorithmic version of Theorem~\ref{thmDoubleR}. For the statement of the lemma and for the remaining subsection, consider an arbitrary graph~\(G_0\) on~\(n_0\) vertices with~\(V(G_0) = [n_0]\) and a tree decomposition~\((T_0,\X_0)\) of~\(G_0\) of width at most~\(t-1\) and with~\(\X_0 = (X_0^i)_{i\in V(T_0)}\). Due to Proposition~\ref{propNonRedundantTD}b), we may assume without loss of generality that~\((T_0,\X_0)\) is nonredundant. We refer to the arrays and the list in Lemma~\ref{lemmaPLabeling} as the \emph{set of \mbox{\(P\)-parameters}}. As the set of \mbox{\(P\)-parameters} will be modified in each iteration, it is convenient that the arrays~\(A_L\), \(A_R\), and~\(A_P\) have length~\(n_0\) and refer to the original name of each vertex in~\(G_0\), even if we consider a graph~\(G \subseteq G_0\) on~\(n < n_0\) vertices. Only the array~\(A_V\) will be indexed by~\([n]\). In particular, compared to the description of the first implementation, we do not set up a bijection that renames the vertices of~\(G\) to~\(1, \ldots, n\). Entries of~\(A_L\), \(A_R\), and~\(A_P\) that refer to vertices not in~\(G\) can be anything. A list of the vertices of~\(G\) can be read off the array~\(A_V\).

\begin{lemma}
  \label{lemmaDoubleRAlgo}
  Let~\(G \subseteq G_0\) be a graph on~\(n\) vertices,~\(m \in [n]\), and let~\(T\) be a tree with~\(V(T)\subseteq V(T_0)\) such that the restriction~\((T,\X)\) of~\((T_0,\X_0)\) to~\(T\) and~\(G\) is a tree decomposition of~\(G\), and let~\(P\subseteq T\) be a nonredundant path. There is an algorithm that receives~\(m\), the set of \mbox{\(P\)-parameters} for~\(G\) with respect to~\((T,\X)\), and~\( (T_0, \X_0)\) as input and computes a list of the vertices in~\(B\) of a cut~\((B,W,Z)\) in~\(G\) that satisfies one of the following properties:
  \begin{enumerate}[1)]
    \item\label{lemmaDoubleRAlgoOption1} \( |B| = m\), \(Z=\emptyset\), \(e_G (B,W) \leq 2 t \Delta(G) \), and the algorithm takes~\(\bigO(n)\) time, or
    \item\label{lemmaDoubleRAlgoOption2} \( |B| \leq m \leq |B| + |Z|\) with \( 0 < |Z| \leq  n/2 \), \(e_G(B,W,Z) \leq  t \Delta(G)  \log ( 16/ w^*(P,\X) )\), and there is a tree~\(T'\) with~\(V(T')\subseteq V(T)\) such that the restriction~\((T',\X')\) of~\((T_0, \X_0)\) to~\(T'\) and~\(G[Z]\) is a tree decomposition of~\(G[Z]\), and~\(T'\) contains a nonredundant path~\(P'\) that satisfies~\(w^*(P',\X') \geq 2 w^*(P,\X)\). In this case, the algorithm modifies the set of \mbox{\(P\)-parameters} to be the set of \mbox{\(P'\)-parameters} for~\(G[Z]\) and takes~\(\bigO\left( n + \sum_{i \in V(T)\setminus V(T')} |X_0^i| \right)\) time. %,  and does not modify the tree decomposition~\( (T_0, \X_0)\). 
  \end{enumerate}
\end{lemma}

Note that the last theorem implies Theorem~\ref{thmDoubleR}. Therefore, the following algorithm will compute a bisection within the bound of Theorem~\ref{thmCutSpecSizes}. First, a heaviest path~\(P_0\) in~\( (T_0,\X_0)\), and the set of \mbox{\(P_0\)-parameters} is computed, which takes~\(\bigO( \|(T_0, \X_0)\| ) \) time, due to Lemma~\ref{lemmaHeaviestPath} and Lemma~\ref{lemmaPLabeling}. As we assumed that~\( (T_0,\X_0)\) is nonredundant,~\(P_0\) is nonredundant. Now, the algorithm in Lemma~\ref{lemmaDoubleRAlgo} is applied to~\(G=G_0\) with~\(T=T_0\), \(P=P_0\), and~\(m = \Bfloor{n_0/2}\). If the computed set~\(B\) contains less than~\(m\) vertices, the algorithm in Lemma~\ref{lemmaDoubleRAlgo} is reapplied to~\(G' := G[Z]\) with the path~\(P'\) and the tree~\(T'\), in the same way as Theorem~\ref{thmDoubleR} was reapplied when proving Theorem~\ref{thmCutSpecSizesExistence}. Using that in each application of Theorem~\ref{thmDoubleR} a different piece of the tree decomposition~\( (T_0,\X_0)\) is used and that~\(n\) shrinks by a factor of at least~\(1/2\) in each round, the desired linear running time can be derived.

\begin{proof}[Proof of Lemma~\ref{lemmaDoubleRAlgo}]
  Let~\(G\), \(n\), \(m\), \((T,\X)\), and~\(P\), be as in Lemma~\ref{lemmaDoubleRAlgo}, let \( \X = (X^i)_{i \in V(T)}\), define~\(V= V(G)\), \(V_T = V(T)\), \(V_P = V(P)\), \(r=w^*(P,\X)\), and let~\(A_L\), \(A_V\), \(A_R\), \(A_P\), and~\(L_P\) be the arrays and the list from the set of \mbox{\(P\)-parameters} for~\(G\). Recall the sets~\(R = \bigcup_{i\in V_P} X^i\) and~\(S = V\setminus R\) from Section~\ref{subsecVertexLabeling}. Unless indicated otherwise, the algorithm follows the construction from Sections~\mbox{\ref{subsecVertexLabeling}-\ref{subsecCase2}}. First, the algorithm checks whether Case~1 of the proof of Theorem~\ref{thmDoubleR} applies, i.e., whether there is a vertex~\(v \in R\) with~\(N_m(v) \in R\). Similarly to the first implementation, this can be done in~\(\bigO(n)\) time, except that the algorithm goes through the array~\(A_V\) to obtain a list of all vertices in~\(G\) and then uses~\(A_R\) to determine a list of vertices in~\(R\), as going through the entire array~\(A_R\), which has length~\(n_0\), might take too long. If Case~1 applies, then the set~\(B\) can be read off the labeling and the algorithm takes~\(\bigO(n)\) time. So from now on, assume that Case~1 does not apply. By Proposition~\ref{propPartitionRiSi}, we have~\(|V_P| \leq n\). Similarly to the first implementation, the algorithm can compute the sizes of the sets~\(|S_i|\), \(|U_i^b|\), \(|H_i^b|\), \(|U_i^f|\), and \(|H_i^f|\) for all~\(i \in V_P\), as well as a node~\(i \in V_P\) with the property of Proposition~\ref{propSpecialNodeForCase2ab} in~\(\bigO(n)\) time. It uses the same trick as above to obtain a list of all vertices in~\(R\) without going through the entire array~\(A_R\). 
  
  Assume the algorithm determined a node~\(i \in V_P\) with the property of Proposition~\ref{propSpecialNodeForCase2ab}a), so it follows the construction described in Section~\ref{subsubsecCase2a}. The case when~\(i\) satisfies Proposition~\ref{propSpecialNodeForCase2ab}b) is analogous. As in the first implementation, the set~\(B_1\) can be read off the labeling in~\(\bigO(n)\) time. To obtain the set~\(B_2\), Lemma~\ref{lemmaApproxCut} is applied to~\(G[S_i]\) with parameters~\( \tilde{m} = m - |B_1|\) and~\(c = 1- \frac{r}{1-r}\). This requires a tree decomposition~\( (\tilde{T},\tilde{\X})\) of~\(G[S_i]\), which can be obtained in a similar way as in the first implementation, i.e., by restricting~\( (T,\X)\) to~\(T\) and~\(G[S_i]\). As~\(v \in S_i\) implies that~\(v \not\in X^i\),  all non-empty clusters of~\( \tilde{\X}\) belong to nodes in~\(V(T_i)\setminus\{i\}\). Hence, it suffices to use~\(\tilde{T} := T_i\) and to compute only the clusters for all~\(h \in V(T_i)\setminus\{i\}\). The tree~\(T_i\) is available from the set of \mbox{\(P\)-parameters}, but the clusters~\(\X\) are not stored. However, the algorithm can use the clusters~\(\X_0\) as the restriction of~\( (T_0,\X_0)\) to~\(\tilde{T}\) and~\(G[S_i]\) is identical to the restriction of~\( (T,\X)\) to~\(\tilde{T}\) and~\(G[S_i]\). So, computing the restriction takes~\( \bigO( \sum_{h \in V(T_i)\setminus\{i\}} |X_0^h|)\) time, as for~\(x\in V(G_0)\) we have~\(x\in V\) if and only if~\(A_L[x] \in [n]\) and~\(A_V[A_L[x]] = x\), and similar techniques as in the first implementation can be applied for checking whether~\(v \in S_i\) for some~\(v\in V\). While doing so, the vertices in~\(\tilde{\X}\) can be renamed to~\(1, \ldots, |S_i|\), which is required for the algorithm in Lemma~\ref{lemmaApproxCut}. The algorithm in Lemma~\ref{lemmaApproxCut} takes \(\bigO( \| (\tilde{T}, \tilde{\X})\| ) = \bigO( \sum_{h \in V(T_i)\setminus\{i\}} |X_0^h| ) \) time and outputs the set~\(B_2\). 
  
  Next, we define the path~\(P'\) and the tree~\(T'\) for the tree decomposition~\( (T',\X')\) of~\(G[Z]\), where~\(Z:= U_i^b \cup H_i^b\). Note that~\(T'\) does not need to be computed, but the algorithm computes the set of \mbox{\(P'\)-parameters}, which depends on~\(T'\). Recall that the ends of~\(P\) were denoted by~\(i_0\) and~\(j_0\), and recall the definition of the set~\(P_i^b\) and the nodes~\(i^b\) and~\(i^b_{\ell}\) from Section~\ref{subsubsecCase2FurtherNotation}. Let~\(H_T\) be the subgraph of~\(T\) that is induced by~\(P_i^b\) and the nodes of~\(T_h\) for all~\(h \in P_i^b\setminus \{i^b\}\). If~\(P_i^b\) induces a connected subgraph in~\(T\), then define~\(T' := H_T\) and~\(i_0' := i^b\). Otherwise~\(P_i^b\) induces two paths in~\(T\), one with ends~\(i^b\) and~\(j_0\) and the other with ends~\(i_0\) and~\(i^b_{\ell}\), and we define~\(T'\) to be the tree obtained from~\(H_T\) by adding the edge~\(\{i^b_{\ell}, i^b\}\) as well as~\(i_0' := i_0\). Furthermore, let~\(P'\) be the path~\(T'[P_i^b]\). Denote by~\( (T',\X')\) the restriction of~\( (T,\X)\) to~\(T'\) and~\(G[Z]\) and note that~\( (T',\X')\) is also the restriction of~\( (T_0,\X_0)\) to~\(T'\) and~\(G[Z]\). In Appendix~\ref{subsecAppLemmaDoubleRAlgo}, we show that~\((T',\X')\) is indeed a tree decomposition of the graph~\(G'\). 
  
  We proceed with the discussion of the relative weight of the path~\(P'\) and the set of \mbox{\(P'\)-parameters} both with respect to~\(\X'\). Consider a \mbox{\(P'\)-labeling} of the vertices in~\(G[Z]\) with respect to the end~\(i_0'\) and note that for all~\(i \in V(P')\) the sets~\(R_i\) will be the same as in the \mbox{\(P\)-labeling}. Therefore,~\(P'\) is nonredundant, \(w(P', \X') = |U_i^b|\), and \(w^*(P',\X') \geq 2r\) because~\eqref{proofDoubleRCase2aH3} implies~\(|Z| \leq |U_i^b|/(2r)\). Before adjusting the set of \mbox{\(P\)-parameters}, a list~\(L_Z\) of the vertices in the set~\(Z\) is computed, which can be done similarly to the first implementation, in~\(\bigO(n)\) time. Recall that the vertices in~\(Z\) are collected in increasing order of their labels, except for possibly one jump from~\(n\) to~\(1\), which can be eliminated. To obtain the new labeling, the algorithm goes through~\(L_Z\) and labels the vertices encountered with~\(1, \ldots, |Z|\), adjusting the arrays~\(A_L\) and~\(A_V\) accordingly, which takes~\( \bigO( |Z|) = \bigO(n)\) time. As the arrays~\(A_R\) and~\(A_P\) refer to the original vertex names used in the graph~\(G_0\), they do not need to be adjusted. To obtain~\(L_{P'}\) from~\(L_{P}\), the algorithm traverses~\(L_Z\), marks all path nodes of the vertices in~\(Z\), and deletes the non-marked nodes from~\(L_P\). As the vertices in~\(Z\) are ordered, this takes \( \bigO (|Z| + |V_P| ) = \bigO(n)\) time. Last but not least, the tree~\(T_{i_0'}\) is modified to consist only of the node~\(i_0'\), which is easy to locate as it is the path-node of the first vertex in~\(L_Z\). 
  
  All in all, each step takes~\(\bigO(n)\) time, except computing the tree decomposition~\( (\tilde{T}, \tilde{\X})\) and the application of Lemma~\ref{lemmaApproxCut}, which each takes~\( \bigO \left( \sum_{h \in V(T_i) \setminus\{i\}} |X_0^h| \right)\) time. As~\(T'\) contains no node of~\(V(T_i)\setminus\{i\}\) due to Proposition~\ref{propCaseIIPart2}d), the desired running time is achieved. 
\end{proof}

\appendix

\section{Omitted Proofs}
\label{secAppendix}

\subsection{Omitted Part in the Proof of Lemma~\ref{lemmaDoubleRAlgo}}
\label{subsecAppLemmaDoubleRAlgo}

In the proof of Lemma~\ref{lemmaDoubleRAlgo}, the pair~\((T',\X')\) was defined. Here we show that~\((T',\X')\) is indeed a tree decomposition of the graph~\(G[Z]\). To show that~\( (T',\X')\) is a tree decomposition of~\(G[Z]\), we use~(T1), (T2), (T3), and~(T3') to refer to the properties of tree decompositions in and after Definition~\ref{defTreeDec}.  As~\(R_j \subseteq X^j\) and~\(S_j \subseteq \bigcup_{h \in V(T_j)} X^h\), the pair~\( (T',\X')\) satisfies~(T1). 

To show that~(T2) is satisfied, let~\(\{x,y\}\) be an arbitrary edge of~\(G[Z]\). If~\(x \not\in R\) then let~\(h_x\) be the path-node of~\(x\) and note that 

  \[I_x \, := \, \left\{h \in V_T : x \in X^h \right\} \, \subseteq \, T_{h_x} \setminus\{h_x\} \, \subseteq \, V(T'),\]

as~\( (T,\X)\) satisfies~(T3'). Since~\( (T,\X)\) satisfies~(T2), there is an~\(h \in I_x\) with~\(y \in X^h\) and hence~(T2) is satisfied for the edge~\(\{x,y\}\). If~\(y\not\in R\), it follows analogously that (T2) is satisfied for the edge~\(\{x,y\}\). So assume that~\(x \in R\) and~\(y \in R\). Denote by~\(h_x\) and~\(h_y\) the path nodes of~\(x\) and~\(y\), respectively, and note that~\(x \in X^{h_x}\) and~\(h_x \in V(P')\) and similarly for~\(y\) and~\(h_y\). If~\(h_x=h_y\) then~(T2) is satisfied for~\(\{x,y\}\). So assume that~\(h_x \neq h_y\) and without loss of generality that~\(h_x\) is encountered before~\(h_y\) when traversing~\(P\) from~\(i_0\) to~\(j_0\). Note that~\(y \not\in X^h\) for all~\(h\) that appear on~\(P\) before~\(h_y\). Therefore,~\(x \in X^{h_y}\) as~\((T,\X)\) satisfies~(T2) and~(T3'), and hence~\((T',\X')\) satisfies~(T2) for~\(\{x,y\}\). 

Clearly, (T3) is satisfied when~\(T' \subseteq T\). So assume that~\(T' \not\subseteq T\). Then, the only edge in~\(T'\) that is not in~\(T\) is the edge~\(e:=\{i^b_{\ell}, i^b\}\). As the unique~\( i^b_{\ell},i^b\)-path in~\(T\) uses no edge in~\(T'\), it follows that~\( (T',\X')\) satisfies~(T3).

\subsection{Implementation of Subroutines}
\label{subsecAppSubroutines}

In what follows, consider a tree decomposition~\((T,\X)\) with~\(T=(V_T,E_T)\) and~\(\X=(X^i)_{i \in V_T}\) of a graph~\(G = (V,E)\) with~\(V=[n]\),  and define~\(n_T := |V_T|\). All procedures receive only~\( (T,\X)\) and are based on a depth-first search in the tree~\(T\).  To describe the adaptations needed in these procedures, we adopt the traditional notation used, for instance, by Cormen et al.~\cite{CLRS09}, of coloring the nodes of the tree white (undiscovered), gray (discovered but not finished), and black (finished) during the traversal, and also the predecessor function, stored in an array~\(\pi\), the discovery time, stored in an array~\(d\), where~\(\pi\) and~\(d\) are both of length~\(n_T\). Furthermore, (T1), (T2), (T3), and~(T3') are used to refer to the properties of tree decompositions in and after Definition~\ref{defTreeDec}.

\subsubsection{Making a Tree Decomposition Nonredundant: Proof of Proposition~\ref{propNonRedundantTD}b)}
\label{subsubsecImplNonRedundantTD}

Here we show that~\( (T,\X)\) can be made nonredundant in~\( \bigO( \|(T,\X)\|)\) time by contracting edges of~\(T\) one by one. More precisely, an edge~\(\{i,j\}\) of~\(T\) is contracted if~\( X^i \subseteq X^j\), and~\(X^j\) is the cluster of the resulting node. If the resulting node is also called~\(j\) and its cluster is~\(X^j\), then we say~\(i\) \emph{is contracted into}~\(j\). Clearly, this satisfies all properties needed for Proposition~\ref{propNonRedundantTD}b).

The algorithm starts a depth-first search at an arbitrary node~\(r\) of~\(T\) and uses an additional array~\(d_G\) of length~\(n\) to store the discovery time of the vertices of~\(G\). Initially all entries of~\(d_G\) are undefined. When~\(i \in V_T\) turns gray, i.e., when it is discovered, the algorithm computes~\(n_i := |X^i|\), and the number~\(c_i\) of vertices in~\(X^i\) which have their discovery time in~\(d_G\) already defined. For~\(i \neq r\) we have~\(c_i = |X^i \cap X^{\pi(i)}|\) due to~(T3'). If~\(i\neq r\) and~\(c_i =n_i\) then~\(X^i \subseteq X^{\pi[i]}\) and the algorithm contracts~\(i\) into~\(\pi[i]\). If~\(i\neq r\), \(c_i \neq n_i\), and \(c_i = n_{\pi[i]}\), then~\(X^{\pi[i]} \subseteq X^i\) and the algorithm contracts~\(\pi[i]\) into~\(i\) and sets~\(d_G[x] = d[\pi[i]]\) for every~\(x\in X^i\) with undefined~\(d_G[x]\). If~\(i=r\) or  none of the above happens, then the algorithm sets~\(d_G[x] = d[i]\) for all~\(x\in X^i\) with~\(d_G[x]\) undefined. In any case the depth-first search proceeds on the modified tree. To contract~\(i\) into~\(\pi[i]\), the algorithm removes~\(\pi[i]\) from the adjacency list of~\(i\) and vice versa. Then it concatenates the adjacency lists of~\(i\) and~\(\pi[i]\) to obtain the new adjacency list of~\(\pi[i]\), updates the corresponding entries of~\(\pi\), exchanges~\(i\) and~\(\pi[i]\) in the adjacency list of each neighbor of~\(i\) (except in the adjacency list of~\(\pi[i]\)), and deletes the node~\(i\) from~\(T\). To contract~\(\pi[i]\) into~\(i\), the algorithm exchanges the links to the clusters of~\(\pi[i]\) and~\(i\), updates~\(n_{\pi[i]}\) to~\(n_i\) as well as~\(d[\pi[i]]\) to~\(d[i]\), and then proceeds as when contracting~\(i\) into~\(\pi[i]\). It is easy to check that the final tree decomposition is a nonredundant tree decomposition of~\(G\).

Next we discuss the running time of the algorithm. The contraction of an edge~\( \{i, \pi[i]\}\) takes time proportional to~\(\deg_T(i)\), when using linked lists, because then it suffices to go through the adjacency list of~\(i\) to find~\(\pi[i]\) and to update all adjacency lists of neighbors of~\(i\) without going through them. Therefore, each original adjacency list is traversed at most once and a modified adjacency list is never traversed. So, in total, all traversals done while contracting an edge take~\(\bigO(n_T)\) time. Also, each cluster of the initial tree decomposition is traversed at most twice when its node is discovered.  Disregarding the time of these traversals, the depth-first search itself is a search of the final decomposition tree, plus constant time per removed node.  So the overall running time is proportional to the size of the input tree decomposition.

\subsubsection{Computing a Heaviest Path: Proof of Lemma~\ref{lemmaHeaviestPath}}
\label{subsubsecImplHeaviestPath}

Here we present an adaptation of a well-known procedure to compute a longest path in a weighted tree~\cite{Handler}. Using a depth-first search in~\(T\) that starts at an arbitrary node~\(r\), the algorithm computes for every~\(i\in V_T\), when~\(i\) turns gray, the value~\(w(r,i) := |\bigcup_{j \in V(Q)} X^j|\), where~\(Q\) is the unique \mbox{\(r\text{,}i\)-path} in~\(T\). For~\(r\), we have~\(w(r,r) = |X^r|\) and, for~\(i \neq r\), we have~\(w(r,i) = w(r,\pi[i]) + |X^i| - c_i\) with~\(c_i = |X^i \cap X^{\pi[i]}|\), due to~(T3'), where~\(c_i\) can be computed as in Section~\ref{subsubsecImplNonRedundantTD}. While doing so, a node~\(s\) such that~\(w(r,s)\) is maximum among all nodes of~\(T\) is recorded. Then, the algorithm does a second depth-first search on~\(T\) that starts at~\(s\) and computes a node~\(t\) such that~\(w(s,t)\) is maximum among all nodes of~\(T\). One can show that the unique \mbox{\(s\text{,}t\)-path} in~\(T\) is a heaviest path in~\(T\) with respect to~\(\X\). The running time of this procedure is~\(\bigO(\|(T,\X)\|)\) as the traversals of~\(T\) take~\(\bigO(n_T)\) time each and the computation of~\(w(r,i)\) and~\(w(s,i)\) for all~\(i\in V_T\) both traverse each cluster in~\(\X\) once.

\subsubsection{Computing a \(P\)-labeling: Proof of Lemma~\ref{lemmaPLabeling}}
\label{subsubsecImplLabeling}

Let~\(P=(V_P, E_P)\subseteq T\) be an arbitrary path and recall the arrays~\(A_L\), \(A_V\), \(A_R\), and \(A_P\), as well as the list~\(L_P\) from Lemma~\ref{lemmaPLabeling}. First, the algorithm initializes all arrays with zeros. Then, it computes the list~\(L_P\) including the trees~\(T_{i}\) for all~\(i\) in~\(L_P\) and sets all entries of~\(A_R\) to one that correspond to vertices in~\(R = \bigcup_{i \in V_P} X^i\), which takes~\(\bigO( \|(T,\X)\|)\) time. Next, the algorithm traverses~\(L_P\) and, for each~\(i\) in~\(L_P\), it traverses the tree~\(T_i\) with a depth-first search, where each node is processed when it turns black.  When processing~\(i\in V_P\), the algorithm labels all unlabeled vertices in~\(X^i\), i.e., fills in the corresponding entries in~\(A_L\) and~\(A_V\), and when processing~\(i \not\in V_P\) it labels all unlabelled vertices in~\(X^i \setminus R\). While doing so, the algorithm also fills in the array~\(A_P\).  The processing time of~\(i\in V_T\) is proportional to~\(X^i\) as the array~\(A_R\) can be used to check whether a vertex~\(v\) is in~\(R\). The overall running time of the procedure is~\(\bigO( \|(T,\X) \|)\).

\subsubsection{Computing an Approximate Cut: Implementation of Algorithm~\ref{algoApproxCut}}
\label{subsubsecImplApproxCut}

In this subsection we present the algorithmic part of the proof of Lemma~\ref{lemmaApproxCut}, i.e., how to implement Algorithm~\ref{algoApproxCut}. In the following, all line numbers refer to Algorithm~\ref{algoApproxCut}. Recall the definition of~\(y_i\) and~\(\tilde{y}_i\) for~\(i\) in~\(V_T\). To compute these values according to~\eqref{proofApproxCutEqYiPart} in Section~\ref{subsecApproxCut}, the algorithm starts a depth-first search in~\(T\) at the node~\(r\) selected in Line~\ref{proofApproxCutPrepro0}. When a node~\(i\) turns gray, the algorithm computes~\(n_i := |X^i|\) and, if~\(i \neq r\), also~\(c_i := |X^i \cap X^{\pi[i]}|\), as done in Section~\ref{subsubsecImplNonRedundantTD}. When a node~\(i\) turns black, the algorithm computes~\(y_i = \sum_{j \text{ child of } i} \tilde{y}_j + n_i\) and, if~\(i \neq r\), also~\(\tilde{y}_i\), which equals~\(y_i - c_i\) by~(T3'). Therefore, Lines~\ref{proofApproxCutPrepro0}-\ref{proofApproxCutPrepro1} take~\( \bigO( \|(T,\X)\|)\) time. 

Before starting the outer while loop, the algorithm sorts the children of each node in order to satisfy the property in Line~\ref{proofApproxCutK}. This is more efficient when all lists are sorted together in one application of Counting Sort~\cite{CLRS09}, which can be applied as~\(\tilde{y}_j \in [n]\) for all~\(j\in V_T\). Each cell carries along a pointer to what list it belongs, so that the algorithm can obtain the separated sorted lists in a way similar to the inverse of the merge step of mergesort. As the number of values to be sorted is at most~\(\sum_{i \in V_T} \deg_T(i) \leq 2n_T\), the sorting takes~\(\bigO(n_T + n) \) time.

This algorithm stores the set~\(B\) as a binary array~\(A_B\), which is initialized with zeros in Line~\ref{proofApproxCutBInitializiation}. This takes~\(\bigO(n)\) time and is possible because of the assumption that~\([n]\) is the vertex set of~\(G\). Every time Line~\ref{proofApproxCutUpdateB} is executed, the set~\(\bigcup_{h=1}^{\ell} \tilde{Y}^{j_h}\) is added to~\(B\), where~\(j_h\) is a child of the current node~\(i\) for~\(h\in[\ell]\). To do so, for each~\(h \in [\ell]\), the algorithm goes through the subtree of~\(T\) rooted at~\(j_h\) and sets the entries of all vertices in the clusters of the nodes discovered there to one, and afterwards, the algorithm sets the values of~\(A_B\) that correspond to vertices in~\(X^i\) back to zero . This works due to~\ref{proofApproxCutInvBY} in Section~\ref{subsecApproxCut}. In Line~\ref{proofApproxCutResetI}, the node~\(i\) is never reset to a descendant of a node~\(j\) for which the set~\(\tilde{Y}^j\) was added to~\(B\), hence each cluster in~\(\X\) is traversed at most once in this process.  Other than in Line~\ref{proofApproxCutUpdateB}, the set~\(B\) is only modified in Line~\ref{proofApproxCutSetZ}, where the set~\(Z\) can be computed greedily by going through the list~\(X^i\). Hence, all modifications on the array~\(A_B\) together take at most~\(\bigO(n + \|(T,\X)\|)\) time, during the entire execution of Algorithm~\ref{algoApproxCut}.

Recall from Section~\ref{subsecApproxCut} that, in the inner while loop, the node~\(i\) is reset to a descendant of one of its children.  Therefore, at most~\(\bigO(n_T)\) time is needed for the inner while loop during the entire algorithm.  Excluding the inner while loop and all modifications on the set~\(B\), the time needed for Lines~\ref{proofApproxCutK}-\ref{proofApproxEndElse} during the entire algorithm is~\(\bigO(n_T)\). Consequently, the overall running time is~\(\bigO(n +n_T + \|(T,\X)\|) = \bigO( \|(T,\X)\|)\) by~(T1). 

\bibliographystyle{abbrv}
\setlength{\itemsep}{-.5ex}
\addcontentsline{toc}{section}{References} 
\bibliography{references}
\addcontentsline{toc}{section}{References} 

\end{document}